\documentclass[12pt, a4paper, reqno]{amsart}
\usepackage{amscd, amssymb}
\usepackage{amsthm}
\usepackage{tikz}
\usetikzlibrary{arrows,positioning} 
\usepackage{mathrsfs}
\usepackage{bbm}
\usepackage{comment}
\usepackage{ stmaryrd }

\usepackage{multirow}
\usepackage{makecell}

\usepackage{hyperref}
\hypersetup{
	colorlinks=true,
	linkcolor=blue,
	filecolor=magenta,  
	urlcolor=cyan,
}

\usepackage{hyperref}
\usepackage[capitalise]{cleveref}

\def\Aut{\operatorname{Aut}}

\def\ker{\operatorname{ker}}

\def\dim{\operatorname{dim}}
\def\ad{\operatorname{ad}}

\def\id{\operatorname{id}}
\def\Dist{\operatorname{Dist}}

\def\ber{\operatorname{ber}}
\def\Ind{\operatorname{Ind}}

\def\Ad{\operatorname{Ad}}

\def\Hom{\operatorname{Hom}}

\def\gr{\operatorname{gr}}

\def\C{\mathbb{C}}
\def\Q{\mathbb{Q}}

\def\N{\mathbb{N}}
\def\Z{\mathbb{Z}}

\def\AA{\mathcal{A}}

\def\UU{\mathcal{U}}

\def\ZZ{\mathcal{Z}}

\def\a{\mathfrak{a}}
\def\b{\mathfrak{b}}
\def\c{\mathfrak{c}}
\def\m{\mathfrak{m}}
\def\n{\mathfrak{n}}
\def\p{\mathfrak{p}}
\def\q{\mathfrak{q}}
\def\g{\mathfrak{g}}
\def\h{\mathfrak{h}}
\def\r{\mathfrak{r}}
\def\k{\mathfrak{k}}
\def\l{\mathfrak{l}}
\def\s{\mathfrak{s}}
\def\o{\mathfrak{o}}
\def\u{\mathfrak{u}}

\newcommand\smallbullet{%
    \raisebox{-0.25ex}{\scalebox{1.2}{$\cdot$}}%
}

\def\ol{\overline}

\def\sub{\subseteq}

\def\xto{\xrightarrow}

\newtheorem{thm}{Theorem}[section]
\newtheorem{cor}[thm]{Corollary}

\newtheorem{lemma}[thm]{Lemma}
\newtheorem{prop}[thm]{Proposition}

\theoremstyle{definition}
\newtheorem{definition}[thm]{Definition}
\theoremstyle{remark}
\newtheorem{remark}[thm]{Remark}
\newtheorem{example}[thm]{Example}

\numberwithin{equation}{section}

\usepackage{relsize}
\usepackage{fancyhdr}
\usepackage[all,cmtip]{xy}

\usepackage{graphicx} 
\usepackage{booktabs} 

\usepackage{tikz}
\usetikzlibrary{positioning,calc,arrows.meta,shapes.geometric,fit,bending}
\usepackage{tikz-3dplot}
\usepackage{subfig}
\usepackage[all,cmtip]{xy}

\usetikzlibrary{decorations.pathmorphing,shapes}

\newif\ifexplanations
\explanationstrue 

\begin{document}
	\title{The Harish-Chandra isomorphism for supersymmetric spaces and ghost distributions}
	
	\author{Shifra Reif, Siddhartha Sahi, Vera Serganova and Alexander Sherman}

\begin{abstract}
We prove the Harish-Chandra isomorphism theorem for supersymmetric spaces, describing the polynomial algebra of eigenvalues of invariant differential operators.  The polynomials obtained satisfy novel invariance conditions, which remain somewhat mysterious.  We also prove the Harish-Chandra isomorphism for ghost distributions, which satisfy a `square root' of the invariance conditions coming from invariant differential operators.  All proofs are algebraic, and rely on a rank-one reduction argument and the Chevalley restriction theorem.

\end{abstract}
    
	\maketitle
	\pagestyle{plain}

\section{Introduction}

The classical Harish-Chandra theorem describes the center of the universal enveloping algebra in terms of ($\rho$-shifted) polynomials in the Cartan subalgebra which are symmetric under the Weyl group.  Harish-Chandra extended this isomorphism to the context of symmetric spaces, where the Cartan subalgebra is replaced by a Cartan subspace, and one uses the corresponding \emph{restricted} root system to obtain an appropriate Weyl group and Weyl vector \cite{H}.  This theorem has extensive applications in differential geometry, integrable systems, special functions, and representation theory.

The Harish-Chandra theorem for (Kac-Moody) Lie superalgebras describing the center of the universal enveloping algebra was given by Sergeev  \cite{Sergeev} and Gorelik-Kac \cite{G2}.  They proved that the center is isomorphic to ($\rho$-shifted) Weyl group invariant polynomials in the Cartan subalgebra satisfying a special `odd invariance' condition $f(\lambda+\alpha)=f(\lambda)$ for $(\lambda+\rho,\alpha)=0$ for all isotropic roots $\alpha$.  Mathematicians  have grappled since with the correct formulation of this theorem, introducing (various incarnations of) the Weyl groupoid, and viewing these polynomials as the Weyl groupoid invariant polynomials in the Cartan subalgebra \cite{SV,SV2}.

\subsection{Invariant differential operators} We extend the work of Sergeev and Gorelik-Kac to supersymmetric spaces for Kac-Moody Lie superalgebras. 
Let $\g$ be a finite-dimensional, complex, indecomposable Kac-Moody Lie superalgebra (for example $\g=\s\l(m|n)$ for $m\neq n$, $\g\l(n|n)$, or $\o\s\p(m|2n))$.  Let $\theta$ be an involution of $\g$ preserving an invariant form on $\g$.
Write $\g=\k\oplus\p$, where $\k$ is the fixed subalgebra and $\p$ is the $(-1)$-eigenspace of $\theta$.  Let $\ZZ_{(\g,\k)}:=(\UU\g/(\UU\g)\k)^{\k}$, where $\k$ acts via left multiplication (or equivalently by the adjoint action). This is an algebra which agrees with the algebra of invariant differential operators on some global supersymmetric space associated to $(\g,\k)$.  

Assume that $(\g,\k)$ admits an Iwasawa decomposition, $\g=\k\oplus \a\oplus\n$.  
Since $(\g_{\ol{0}},\k_{\ol{0}})$ will be a symmetric pair in the classical sense, we may choose a Cartan subspace $\a\sub\p_{\ol{0}}$ which is unique up to conjugacy.  Then after choice of positive restricted roots (see below), we obtain a Harish-Chandra projection
\[
HC:\UU\g/(\UU\g)\k\to S(\a),
\]
which is an algebra homomorphism when restricted to $\ZZ_{(\g,\k)}$.  We aim to describe $HC(\ZZ_{(\g,\k)})$, which will in particular describe the spectrum of invariant differential operators.

To obtain a description, consider the restricted root system $\Delta\sub\a^*$ of $(\g,\k)$, which is given by the set of non-zero restrictions of roots of $\g$ to $\a$. The restricted root system of $(\g_{\ol{0}},\k_{\ol{0}})$ will sit inside $\Delta$ from which we obtain a (baby) Weyl group $W\sub \Aut(\a)$. To every (restricted) root $\alpha\in\Delta$, we have a root space $\g_{\alpha}\sub\g$, and we set 
\[
m_{\alpha,0}:=\dim(\g_{\alpha})_{\ol{0}}, \ \ \ m_{\alpha,1}:=\dim(\g_{\alpha})_{\ol{1}}, \ \ \ n_{\alpha}:=m_{\alpha,1}/2,
\]
(note that $m_{\alpha,1}$ is always even, so $n_{\alpha}\in\N$).  We say that $\alpha$ is \emph{regular} (resp.~\emph{singular}) if $\g_{\alpha},\g_{-\alpha}$ generate a subalgebra containing a copy $\s\l(2)$ (resp.~$\g_{\pm\alpha}$ are purely odd).  One can show that every (restricted) root is either regular or singular.   Finally, make a choice of positive restricted roots $\Delta^+\sub\Delta$, from which we obtain a Weyl vector $\rho\in\a^*$ defined as:
\[
\rho=\frac{1}{2}\sum\limits_{\alpha\in\Delta^+}(m_{\alpha,0}-m_{\alpha,1})\alpha.
\]
We will write $S(\a)^{W_{\smallbullet}}$ for the $W$-invariant polynomials under the $\rho$-shifted action of $W$ on $S(\a)$ (the dot action).

Define $J_{\a}\sub S(\a)$ to consist of those $f\in S(\a)^{W_{\smallbullet}
}$ such that for any (restricted) root $\alpha$ and any $\lambda\in\a^*$ satisfying $(\lambda+\rho,\alpha)=0$,
\begin{equation}\label{eqn sing invc cond}
    f(\lambda+r\alpha)=f(\lambda-r\alpha), \quad  1\leq r\leq n_{\alpha}.
\end{equation}

\begin{thm}\label{main thm}
    The Harish-Chandra projection gives an isomorphism of algebras
    \[
    \ZZ_{(\g,\k)}\xto{\sim}J_{\a}.
    \]
\end{thm}

The authors are not aware of the peculiar invariance condition (\ref{eqn sing invc cond}) appearing elsewhere in the literature. 
A description of $HC(\ZZ_{(\g,\k)})$ was also given in \cite{A1} using techniques from integration theory on CS manifolds, but no explicit invariance conditions were given to describe the polynomials in $HC(\ZZ_{(\g,\k)})$ in \emph{loc.~cit.}  Moreover, our proof is purely algebraic and more in the spirit of Lepowsky's arguments in \cite{L1,L2}.

\subsection{Ghost distributions}

In the super setting there is a parallel story to the above, which began with the introduction of the sCasimir (or Casimir's ghost) in \cite{ABL}, an element in the enveloping algebra of $\o\s\p(1|2n)$ which squares to a central element. In \cite{G1} Gorelik generalised the work of \cite{ABL} with the introduction of the anti-center $\AA$ and ghost center $\widetilde{\ZZ}$, which lie in the enveloping algebra of an arbitrary Lie superalgebra, and work as a useful enlargement of the center.  For $\g$ Kac-Moody, Gorelik described the Harish-Chandra image of $\AA$ and $\widetilde{\ZZ}$.  

In \cite{Sh1}, Gorelik's theory of the ghost center was generalized to the setting of supersymmetric spaces as follows.  If $(\g,\k)$ is a supersymmetric pair, we consider the new supersymmetric pair $(\g,\k')$ arising from the involution $\delta\circ\theta$, where $\delta(x)=(-1)^{\ol{x}}x$ is the grading operator.  The Lie superalgebra $\k'$ acts naturally on the distributions of $G/K$, i.e.~$\UU\g/\UU\g\k$,  via left multiplication and we write $\AA_{(\g,\k)}$ for the invariants of this action.  We call the elements of $\AA_{(\g,\k)}$ \emph{ghost distributions}.  It is always the case that $\AA_{(\g,\k)}$ is a module over $\ZZ_{(\g,\k)}$, and the same will be true of its Harish-Chandra image.  The Harish-Chandra image of $\AA_{(\g,\k)}$ contains information about the following branching problem: when does a spherical representation $V$ of $\g$ contain a copy of $I_{\k'}(\C)$, the indecomposable injective on the trivial module for $\k'$.

In many cases of interest, it happens that $(\g,\k')$ also satisfies the Iwasawa decomposition, in which case it is conjugate to $(\g,\k)$ and we call the pairs \emph{interlacing}.  When $(\g,\k)$ is interlacing, it was shown in \cite{Sh2} that $\ZZ_{(\g,\k)}\oplus\AA_{(\g,\k)}$ has the structure of an algebra such that the Harish-Chandra projection is an algebra homomorphism, generalizing Gorelik's ghost center.  
  
We now provide a description of $HC(\AA_{(\g,\k)})$ with respect to a chosen base $\Sigma\sub\Delta$.  For a restricted root $\alpha\in\Delta$, we say that $\alpha$ is \emph{odd} if it is the restriction of an odd root from $\g$, and otherwise we say it is \emph{purely even}.  Write $\Delta_{odd}$ for the set of odd roots, and $\Delta_{ev}$ for the purely even roots.  Define the \emph{pure Weyl vector} to be :
\[
    \rho^{\text{pure}}=\frac{1}{2}\sum\limits_{\alpha\in\Delta_{ev}^+}m_{\alpha,0}\alpha-\frac{1}{2}\sum\limits_{\alpha\in\Delta_{odd}^+}m_{\alpha,1}\alpha.
    \]
We have $\rho^{\text{pure}}=\rho$ if and only if $(\g,\k)$ is interlacing. 
Consider the following polynomial $T_{\Sigma}\in S(\a)$, where for a root $\alpha$ we write $h_{\alpha}(\lambda):=(\alpha,\lambda)$:
\[
T_{\Sigma}:=T_{reg}T_{iso},
\]
where
\[
T_{reg}:=\prod\limits_{\alpha\in\Delta_{reg}^+}\prod\limits_{s=0}^{n_{\alpha}-1}(h_{\alpha}+(\rho^{\text{pure}}+(1-n_{\alpha}+2s)\alpha,\alpha)),
\]
\[
T_{iso}:=\prod\limits_{\alpha\in\Delta^+_{iso}}(h_{\alpha}+(\rho,\alpha)).
\]
Then define $M_{\a}\sub S(\a)$ to consist of those polynomials $f\in S(\a)^{W_{\smallbullet}}T_{\Sigma}$ such that for admissible singular roots $\alpha$ (Definition \ref{defn admissible roots}), we have
\begin{equation}\label{eqn odd ghost invce}
    f(\lambda+r\alpha)=(-1)^{r}f(\lambda-r\alpha), \  \ \ 1\leq r\leq n_{\alpha},
\end{equation}
for $(\lambda+\rho,\alpha)=0$.  
\begin{thm} \label{main ghost theorem} 
The Harish-Chandra projection gives an isomorphism
\[
\AA_{(\g,\k)}\xto{\sim} M_{\a}.
\]
\end{thm}
We note that $M_\a$ is a $J_\a$-module.  Further, if $\rho=\rho^{\text{pure}}$ (i.e.~$(\g,\k)$ is interlacing), then $T_{\Sigma}^2$ is $W_{\smallbullet}$-invariant, and the product of two polynomials in $M_{\a}$ lies in $J_{\a}$. 

\subsection{An invariant Harish-Chandra homomorphism}  The Harish-Chandra projection $HC$ depends on a choice $\Delta^+$ of positive restricted roots, and this is equivalent to a choice $\Sigma$ of simple roots.  For this subsection, let us make this dependency explicit by writing $HC_{\Sigma}$ in place of $HC$.  Already for $\ZZ_{(\g,\k)}$ and $\AA_{(\g,\k)}$, we see a dependency on the Weyl vector in the image.  In an attempt to construct a single Harish-Chandra homomorphism that lacks a dependency on $\Sigma$, let us introduce $hc_{\Sigma}=t_{\rho}^*\circ HC_{\Sigma}:\UU\g/(\UU\g)\k\to S(\a)$, where $t_{\rho}^*(f)(\mu)=f(\mu-\rho)$.  Then we see that $hc_{\Sigma}(\ZZ_{(\g,\k)})$ is the set of polynomials $f\in S(\a)^W$ such that (\ref{eqn sing invc cond}) holds whenever $(\lambda,\alpha)=0$.  As a consequence of Proposition \ref{prop: reflections on HC}, we can moreover show that when restricted to $\ZZ_{(\g,\k)}$ we have $hc_{\Sigma}
=hc_{\Sigma'}$ for any two positive systems $\Sigma,\Sigma'$.  Thus we have a `well-defined' Harish-Chandra homomorphism $\ZZ_{(\g,\k)}\to S(\a)$.

For ghost distributions, the image $hc_{\Sigma}(\AA_{(\g,\k)})$ will depend on $\Sigma$ in general due to the (somewhat awkward) dependence on $\rho^{\text{pure}}$ and the set of admissible roots (see Example \ref{ex:non interlacing non invariance}).  However in the interlacing case,  $hc_{\Sigma}(\AA_{(\g,\k)})$ will have a uniform description as the set of polynomials $f\in S(\a
)$ which are anti-invariant (without $\rho$-shift) under the action of $W$ (see Lemma \ref{lemma interlacing invce conditions}), and satisfy (\ref{eqn odd ghost invce}) whenever $(\lambda,\alpha)=0$.  By Proposition \ref{prop ghost borel change}, we have $hc_{r_{\alpha}\Sigma}=(-1)^{n_{\alpha}}hc_{\Sigma}$ for any singular root $\alpha\in\Sigma$, while $hc_{r_\alpha\Sigma}=hc_{\Sigma}$ whenever $\alpha\in\Sigma$ is a regular root.  It follows (by Lemma \ref{lemma all Sigma conj}) that up to sign we obtain a well-defined Harish-Chandra homomorphism $\AA_{(\g,\k)}\to S(\a)$.

\subsection{Methods of Proof}
The proofs of Theorems \ref{main thm} and \ref{main ghost theorem} consist of two main steps. The first is to show that the image of the Harish-Chandra projection satisfies the required invariance conditions. The second is to prove that the projection maps are bijective.

To accomplish the first step, we established that the image of the Harish-Chandra map satisfies the desired invariance condition by restricting the pair to rank-one subalgebras corresponding to each root $\alpha$.  For Theorem \ref{main thm} it is done using that the center surjects onto the invariant differential operators for rank-one symmetric pairs. For Theorem \ref{main ghost theorem}, it was proven by \cite[Sec. 9]{Sh2} using representation-theoretic properties for the relevant rank-one pairs.
This gives the invariance conditions for simple roots. 
We then analyze how the Harish-Chandra map transforms when changing the choices of simple roots, which allows us to extend these properties from simple roots to all admissible roots. 
Having shown that the algebras of invariant polynomials can be defined using conditions corresponding to admissible roots, we complete this step of the proof. 

To prove that the maps are bijective, we use the associated graded spaces and compare the graded dimensions. The associated graded algebra of the algebra of invariant distributions is the algebra $S(\p)^\k$, shown to be isomorphic to the algebra $I(\a)$ in the Chevalley restriction theorem \cite{A2,RSS}. We show that the associated graded algebra of $J_\a$ is also $I(\a)$.
A parallel argument holds for ghost distributions, where we show that the associated graded space of $\mathcal{A}_{(\g,\k)}$ (as computed in \cite{Sh2}) is isomorphic to that of $M_\a$.

\subsection{Outline of paper} 
The paper is organized as follows.
In \cref{Sec:setup and restricted roots}, we review preliminaries about supersymmetric spaces and restricted root systems. This in particular includes properties of the Weyl group and its action, choices of simple roots, the notion of admissible roots, and properties of interlacing pairs. 

In  \cref{Sec: invariant polynomials} we define two subspaces of invariant polynomials. We will later show that they are precisely the  images of the  invariant distributions and ghost distributions under the Harish-Chandra projection. 

\cref{Sec:HC proj and rank one} introduces the  vector space of distributions, the Harish-Chandra projection map, and its behavior upon restriction to rank-one subalgebras.
In \cref{Sec: invariant distributions}, we prove \cref{main thm} and in \cref{Sec: Ghost proof} we discuss ghost distributions and prove \cref{main ghost theorem}.

   Finally, in an appendix we prove that when 
$\s\mathfrak t\r(\operatorname{ad}x)=0$ for any $x\in\k$ then \linebreak $(\UU\g)^\k\cap\k\UU\g=(\UU\g)^\k\cap\UU\g\,\k$. This generalizes a classical result for Lie algebras (see for example, \cite[Prop 9.1.10(ii)]{D}). 


\subsection{Acknowledgements}  The authors would like to thank Hadi Salmasian for interesting discussions.  The project was made possible by a SQuaRE at the American Institute for Mathematics. The authors thank AIM for providing a supportive and mathematically rich environment.
The research of R.S. was partially supported by  Israel Science Foundations Grants No. 1957/21.
The research of S.S. was partially supported by the
Simons Foundation under grant number 00006698. V.S. was supported by the NSF-BSF grant 2019694. A.S. was supported by ARC grant DP210100251 and by an AMS-Simons Travel Grant.

\section{Supersymmetric pairs and Restricted root systems}
\label{Sec:setup and restricted roots}

For background on Kac-Moody superalgebras, we refer the reader to \cite{S4}.

\subsection{Supersymmetric pairs}
Let $\g$ be a Kac-Moody superalgebra with a chosen nondegenerate, invariant form $(-,-)$.  Let $\theta$ be an involution of $\g$  which preserves $(-,-)$.  Set $\k$ to be the fixed-point subalgebra of $\theta$ and write $\p$ for the $(-1)$-eigenspace.  Let $\a\sub\p_{\ol{0}}$ be a Cartan subspace, meaning a maximal even subalgebra of $\p$ which consists only of semisimple elements.  Note that a Cartan subspace is abelian and unique up to conjugation by inner automorphism from $\k_{\ol{0}}$ acting on $\g_{\ol{0}}$ (\cite[Lem.~26.15]{timashev}).  We extend $\a$ to a $\theta$-stable Cartan subalgebra $\h$ of $\g$, and write $\h=\a\oplus\mathfrak{t}$ where $\mathfrak{t}:=\k\cap\h$.

Since $\theta$ preserves the form, notice that $(-,-)$ will descend to a nondegenerate form on $\a$, and thus also on $\a^*$.  As usual, we will call an element $\alpha\in\a^*$ \emph{isotropic} if $(\alpha,\alpha)=0$, and otherwise we will say that $\alpha$ is \emph{non-isotropic}.

\subsection{Restricted root systems} Write $\Delta\sub\a^*\setminus\{0\}$ for the set of nonzero weights for the action of $\a$ on $\g$.  It is easy to see that $\Delta$ is obtained by taking the nonzero restrictions of roots of $\g$ (in $\h^*$) to $\a^*$.  We call $\Delta$ the \emph{restricted root system} of $(\g,\k)$, and elements $\alpha\in\Delta$ are called \emph{restricted roots}.
 
\begin{definition}
    For a restricted root $\alpha\in\Delta$, define its multiplicity $m_{\alpha}$ to be the pair $(m_{\alpha,0}|m_{\alpha,1})$, where 
    \[
    m_{\alpha,i}:=\dim(\g_{\alpha})_{\ol{i}}.
    \]
\end{definition}

Note that $m_{\alpha,1}$ is always an even number due to the pairing $\langle x,y\rangle=(\theta(x),y)$ being a symplectic form on $(\g_{\alpha})_{\ol{1}}$, \cite[Prop. 2.10]{A2}.  Because of this, we will write 
\[
n_{\alpha}:=m_{\alpha,1}/2\in\N.
\]

\begin{definition}
 A restricted root $\alpha\in\Delta$ is called \emph{singular} if $m_{k\alpha,0}=0$ for all $k\in\Q\setminus\{0\}$, and \emph{regular} otherwise.
\end{definition}

Let $\phi:\Z\Delta\to \mathbb{R}$ be a group homomorphism such that $\phi(\alpha)\neq0$ for all $\alpha\in\Delta$.  Then set
\[
\Delta^{\pm}_{\phi}:=\{\alpha\in\Delta:\pm\phi(\alpha)>0\}\sub\Delta.
\]
It is clear that $\Delta=\Delta_{\phi}^+\sqcup\Delta_{\phi}^-$.  

\begin{definition}
    A \emph{positive system} is a choice of subset $\Delta^+\sub\Delta$ such that $\Delta^+=\Delta^+_{\phi}$ for some $\phi$.
\end{definition}

\begin{definition}
    We say $\Sigma\sub\Delta$ is a simple system (or a base, or a choice of simple roots) if $\Sigma$ is linearly independent in $\a^*$ and
    \[
    \Delta\sub\pm\sum\limits_{\alpha\in\Sigma}\N\alpha.
    \]
\end{definition}

\begin{lemma}
        \begin{enumerate}
            \item For a positive system $\Delta^+$, let $\Sigma\sub\Delta^+$ be the set of indecomposable roots.  Then $\Sigma$ is a simple system.
            \item For a simple system $\Sigma$, we obtain a positive system by setting:
            \[
\Delta^{+}=\Delta\cap\left(\sum\limits_{\alpha\in\Sigma}\N\alpha\right).
            \]
        \end{enumerate}
        Further, the above correspondence is a bijection.
\end{lemma}
\begin{proof}
    (1) follows from Lem.~5.10 of \cite{Sh4}, and (2) is clear.  The fact that the correspondence is bijective is Prop.~5.11 of \cite{Sh4}.
\end{proof}

In what follows, for a simple system $\Sigma\sub\Delta$, we write $\Delta^+_{\Sigma}:=\Delta\cap\left(\sum\limits_{\alpha\in\Sigma}\N\alpha\right)$ for the corresponding positive system.

\subsection{Little Weyl group} \label{sec:little Weyl group}

\begin{definition}
    For a regular root $\alpha\in\Delta$ define the reflection $r_{\alpha}\in GL(\a^*)$ to be
    \[
    r_{\alpha}(\lambda)=\lambda-2\frac{(\lambda,\alpha)}{(\alpha,\alpha)}\alpha.
    \]
\end{definition}

\begin{definition}
    Define the little Weyl group $W$ of $(\g,\k)$ with respect to $\a$ to be the subgroup of $GL(\a^*)$ generated by the reflections $r_{\alpha}$ for all regular roots $\alpha\in\Delta$.
\end{definition}

\begin{lemma}
    \begin{enumerate}
        \item The little Weyl group $W$ of $(\g,\k)$ is equal to the little Weyl group of $(\g_{\ol{0}},\k_{\ol{0}})$.
        \item We have an isomorphism
        \[
        N_{K}(\a)/C_{K}(\a)\cong W.
        \]
    \end{enumerate}
\end{lemma}
\begin{proof}
    The first claim is obvious, and the second claim is classical (Prop.~26.19, \cite{timashev}).
\end{proof}

\subsection{Singular reflections}\label{sec reflections}

    For a simple system $\Sigma$ and a singular root $\alpha\in\Sigma$, we define $r_{\alpha}\Sigma$ to be the simple system corresponding to the positive system $(\Delta^+_{\Sigma}\setminus\{\alpha\})\cup\{-\alpha\}$.  

    We call the correspondence $\Sigma\mapsto r_{\alpha}\Sigma$ a \emph{singular reflection}.

\subsection{Iwasawa decomposition}\label{section iwasawa} Write $\m=\c(\a)\cap\k$.   We impose the following \textbf{assumption} for the rest of the paper:
	
\begin{equation}\label{assumption iwasawa}
    \c(\a)=\m\oplus\a.
\end{equation}  

For a simple system $\Sigma,$ we define
\[
\n_{\Sigma}:=\bigoplus\limits_{\alpha\in\Delta_{\Sigma}^+}\g_{\alpha}.
\]
We will abbreviate $\n_{\Sigma}$ by $\n$ when the context is clear.  Set $\q=\m\oplus\a\oplus\n$; notice that this is a parabolic subalgebra with Levi subalgebra $\c(\a)$.   By our assumption (\ref{assumption iwasawa}), $(\g,\k)$ admits an Iwasawa decomposition:
 \begin{equation}\label{iwasawa decomp}
      \g=\k\oplus\a\oplus\n.
 \end{equation}
 
 In particular we have a decomposition 
\begin{equation}\label{decomp of Ug from Iwasawa}
\UU\g=S(\a)\oplus(\n\UU\g+(\UU\g)\k).    
\end{equation}	

\subsection{Rank-one subalgebra} \label{subsec:rank one prelim}   For $\alpha\in\Delta$ we write: 
	\[
	\g(\alpha)=\c(\a)\oplus\bigoplus\limits_{k\in\Q\setminus\{0\}}\g_{k\alpha}.
	\]	
	Then $\g(\alpha)$ is again Kac-Moody (\cite[Lem.~3.1]{SSh}), is $\theta$-stable, and we have a decomposition $\g(\alpha)=\k(\alpha)\oplus\p(\alpha)$.  If we let $\n(\alpha):=\g_{\alpha}\oplus\g_{2\alpha}$, then our Iwasawa decomposition (\ref{assumption iwasawa}) persists, i.e. $\g(\alpha)=\k(\alpha)\oplus\a\oplus\n(\alpha)$.
    
    Assume further that we have chosen a base $\Sigma\sub\Delta^+$ and that $\alpha\in\Sigma$, and write
    \begin{equation}\label{eqn nalpha}    \n^{\alpha}:=\bigoplus\limits_{\beta\in\Delta^+\setminus\{\alpha,2\alpha\}}\g_{\beta},
    \end{equation}
    so that $\n=\n(\alpha)\oplus\n^{\alpha}$.  Choose a $\k(\alpha)$-stable complement $\r$ of $\k(\alpha)$ in $\k$.  Then we have
	\begin{equation}
	    \label{eq:decomp with the rank one subalg}\g=\r\oplus\k(\alpha)\oplus\a\oplus\n(\alpha)\oplus\n^{\alpha}=\r\oplus\g(\alpha)\oplus\n^{\alpha}.
	\end{equation}
	Observe that since $\alpha$ is simple, $\n^{\alpha}$ is $\g(\alpha)$-stable, and thus $\k(\alpha)$-stable, meaning the above decomposition is $\k(\alpha)$-stable.  It follows that we obtain a $\k(\alpha)$-stable decomposition 
	\begin{equation}\label{eqn ug decomp at alpha}
    \UU\g=\UU\g(\alpha)\oplus((\UU\g) \r+\n^{\alpha}\UU\g).
	\end{equation}

\subsection{Weyl vector}
\begin{definition}
    For a choice of simple roots $\Sigma\sub\Delta$, define the Weyl vector $\rho_{\Sigma}$ by
    \[
    \rho=\rho_{\Sigma}:=\frac{1}{2}\sum\limits_{\alpha\in\Delta^+}(m_{\alpha,0}-m_{\alpha,1})\alpha.
    \]
    Further, for $\alpha\in\Sigma$ we set $\rho_{\alpha}$ to be the Weyl vector for $(\g(\alpha),\k(\alpha))$ with respect to the positive system $\{\alpha\}$.
\end{definition}

\begin{lemma}\label{lemma rho refl}
    Let $\Sigma\sub\Delta$ be a choice of simple roots, and let $\alpha\in\Sigma$.
    \begin{enumerate}
        \item If $\alpha$ is regular then:
        \[
        \rho_{r_{\alpha}\Sigma}=r_{\alpha}(\rho_{\Sigma}).
        \]
        \item If $\alpha$ is singular then:
        \[
        \rho_{r_{\alpha}\Sigma}=\rho_{\Sigma}+m_{\alpha,1}\alpha.
        \]
    \end{enumerate}
\end{lemma}
\begin{proof}
    Follows from direct computation.
\end{proof}

\begin{lemma}\label{lemma local global rho}
    For a simple root $\alpha$, we have
    \[
    (\rho,\alpha)=(\rho_{\alpha},\alpha)=\frac{1}{2}(m_{\alpha,0}+2m_{2\alpha,0}-m_{\alpha,1})(\alpha,\alpha).
    \]
\end{lemma}

\begin{proof}
 The second equality is clear by definition.  Thus it suffices to check that $(\rho-\rho_{\alpha},\alpha)=0$.  If $\alpha$ is regular, then $r_{\alpha}$ preserves $\Delta^+\setminus(\Delta^+\cap\N\alpha)$.  Thus $r_{\alpha}(\rho-\rho_{\alpha})=\rho-\rho_{\alpha}$, and we have 
 \[
 (\rho-\rho_{\alpha},\alpha)=(r_{\alpha}(\rho-\rho_{\alpha}),r_{\alpha}(\alpha))=-(\rho-\rho_{\alpha},\alpha),
 \]
 and thus $(\rho-\rho_{\alpha},\alpha)=0.$
 
 Now let us deal with the case when $\alpha$ is singular.  For this, recall that $\g(\alpha)$ acts on $\n^{\alpha}$ (\ref{eqn nalpha}).  Let $h_{\alpha}\in\a$ denote the coroot of $\alpha$, so that $(\lambda,\alpha)=\lambda(h_{\alpha})$ for all $\lambda\in\a^*$.  Then $h_{\alpha}\in\g(\alpha)$, and $(\rho-\rho_\alpha,\alpha)$ is exactly the supertrace of $h_{\alpha}$ acting on $\n^{\alpha}$. It therefore suffices to prove that $h_{\alpha}$ lies in the derived subalgebra of $\g(\alpha)$.  In fact, it suffices to prove $h_{\alpha}$ lies in $[\g_{\alpha},\g_{-\alpha}]$.  However by the classification of rank-one pairs with a singular simple root ((2) of Lemma \ref{lemma classification rank one}), we see that the subalgebra generated by $\g_{\alpha}$ and $\g_{-\alpha}$ is either $\o\s\p(2|2n)$ or $\g\l(1|1)\times\g\l(1|1)$.  In the first case $\o\s\p(2|2n)$ is its own derived subalgebra, and the second case our claim can be checked directly.

\end{proof}

\subsection{Odd roots and the pure Weyl vector}
\begin{definition}
    We say that $\alpha\in\Delta$ is \emph{odd} if it is the restriction of an odd root of $\g$, and write $\Delta_{odd}$ for the subset of odd roots.  We set $\Delta_{ev}:=\Delta\setminus\Delta_{odd}$, and call elements of $\Delta_{ev}$ \emph{purely even} roots.  Finally, write $\Delta_{reg,odd}\sub\Delta$ for the set of roots that are simultaneously odd and regular.
\end{definition}

Notice that if $\alpha$ is purely even, then $\g_{\alpha}$ must be a purely even vector space.  However it need not be true that $\g_{\alpha}$ is purely odd whenever $\alpha$ is odd.

\begin{definition}
    We say that the restricted root system $\Delta\sub\a^*$ is \emph{pure} if for all $\alpha\in\Delta$, $m_{\alpha,0}m_{\alpha,1}=0$.
\end{definition}

\begin{definition}
    For a base $\Sigma\sub\Delta$, define the \emph{pure} Weyl vector $\rho^{\text{pure}}_{\Sigma}$ of $(\g,\k)$ to be:
    \[
    \rho^{\text{pure}}_{\Sigma}=\frac{1}{2}\sum\limits_{\alpha\in\Delta_{ev}^+}m_{\alpha,0}\alpha-\frac{1}{2}\sum\limits_{\alpha\in\Delta_{odd}^+}m_{\alpha,1}\alpha.
    \]
\end{definition}
Notice that $\rho^{\text{pure}}_{\Sigma}$ is the Weyl vector of the pure restricted root system obtained from $\Delta$ by setting $m_{\alpha,0}=0$ for all $\alpha\in\Delta_{odd}$.

\begin{lemma}\label{lemma pure rho}
    Let $\Sigma\sub\Delta$ be a chosen base, and let $\alpha\in\Sigma$.
    \begin{enumerate}
        \item If $\alpha$ is regular, then 
        \[
        \rho_{r_{\alpha}\Sigma}^{\text{pure}}=r_{\alpha}(\rho^{\text{pure}}_{\Sigma}).
        \]
        \item If $\alpha$ is singular, then 
        \[
        \rho_{r_{\alpha}\Sigma}^{\text{pure}}=\rho^{\text{pure}}_{\Sigma}+m_{\alpha,1}\alpha.
        \]
    \end{enumerate}
\end{lemma}
\begin{proof}
    Follows from direct computation.
\end{proof}

\subsection{Conjugacy properties of simple root systems}  

    The following is \cite[Lem.~2.22]{Sh3}.

    \begin{lemma}\label{lemma all Sigma conj}
        Any two simple root systems $\Sigma,\Sigma'$ may be obtained from one another by a sequence of singular reflections and conjugation by $W$. In particular, $|\Sigma|=|\Sigma'|$.  
    \end{lemma}

    \begin{definition}
    The rank of $(\g,\k)$ is defined to be the cardinality of any simple root system, which is well-defined by Lemma \ref{lemma all Sigma conj}.
\end{definition}


    Write $\Delta_{reg}\sub\Delta$ for the set of regular roots, and $\Delta_{sing}$ for the set of singular roots.  Then $\Delta_{reg}$ forms a (potentially non-reduced) root system in the classical sense.   For a base $\Sigma\sub\Delta$, let $\Pi\sub\Delta_{reg}^+$ denote the simple roots of $\Delta_{reg}^+$.  Then the simple reflections in elements of $\Pi$ generate the Weyl group $W$.  Write $W'$ for the subgroup of $W$ generated by the simple reflections $r_{\alpha}$ in purely even roots $\alpha\in\Pi$.  

    \begin{definition}\label{defn admissible roots}
        We say that a root $\alpha\in\Delta_{\Sigma}^+$ is \emph{admissible} with respect to $\Sigma$ if $\alpha$ becomes simple after application of some number of singular reflections to $\Sigma$.  
    \end{definition}
    Write $\Delta_{\Sigma,ad}\sub\Delta_{\Sigma}^+$ for the set of admissible roots.  Note that for $w\in W$, we have $w(\Delta_{\Sigma,ad})=\Delta_{w\Sigma,ad}$.

    \begin{lemma}\label{lemma conj props}
        \begin{enumerate}
            \item We have $\Pi\sub\Delta_{\Sigma,ad}$.
            \item We have $\Delta_{reg,odd}^+\sub W'(\Pi\cap\Delta_{reg,odd})$.
            \item $\Delta_{sing}\sub\pm W(\Delta_{\Sigma,ad,sing})$. 
        \end{enumerate}
    \end{lemma}
    \begin{proof}
        (1) is proven in \cite[Lem.~2.24]{Sh3}.  For (3), let $\alpha\in\Delta_{sing}$.  We may use the proof of \cite[Lem.~2.23]{Sh3} to show that $\alpha$ lies in some base $\Sigma'$.  By Lemma \ref{lemma all Sigma conj}, $\Sigma$ may be obtained from $\Sigma'$ by a sequence of simple reflections $r_{\alpha_1},\dots,r_{\alpha_k}$.   Write $w\in W$ for the composition of the reflections $r_{\alpha_i}$ where $\alpha_i$ is regular. Then we obtain that $\pm w\alpha\in\Delta_{\Sigma,ad}$.  


        For (2), we apply the classification of supersymmetric pairs, given in table of \cite[\S 5]{Sh4} and relying on \cite{S} and the appendix of \cite{Sh4}.  By inspection, the only restricted root systems with odd regular roots are subsystems of $BC(r,s)$ which contain $A(r,s)$, and for which the odd regular roots are the short roots of $BC_r$ and $BC_s$.  Thus we have $S_r\times S_s\sub W'$, and in both $BC_r$ and $BC_s$ there is a unique $W'$-orbit on the odd regular roots, up to sign.  By \cite[Lem.~2.23]{Sh3}, we obtain that $\Pi\cap\Delta_{reg,odd}$ must contain a short root in each component, and from this we obtain (2). 

    \end{proof}

    We note a corollary of the above proof that we use later:
    \begin{cor}\label{cor reg odd roots ortho}
        If $\alpha,\beta\in\Delta_{reg,odd}$, then either $\alpha=\pm\beta$ or $(\alpha,\beta)=0$.
    \end{cor}

\begin{example}\label{osp example}
    Consider the case $\g=\o\s\p(4+a|8+2b)$, $\k=\o\s\p(2|4)\times\o\s\p(2+a|4+2b)$. Then we may give $\a$ a basis $\varepsilon_1,\varepsilon_2,\delta_1,\delta_2$ such that $(\varepsilon_i,\varepsilon_i)=-2(\delta_i,\delta_i)$. The roots with their multiplicities may be described as follows:
    \begin{center}
\begin{tabular}{c|c|c|c|c|c|c}

{Root} & $\pm\varepsilon_1\pm\varepsilon_2$ & $\pm\varepsilon_1,\pm\varepsilon_2$ & $\pm\delta_1\pm\delta_2$ & $\pm\delta_1,\pm\delta_2$ & $\pm2\delta_1,\pm2\delta_2$ & $\pm\varepsilon_i\pm\delta_j$ \\
\hline
{Multiplicity} & $(1|0)$ & $(a|2b)$ & $(4|0)$ & $(4b|2a)$ & $(3|0)$ & $(0|2)$ \\

\end{tabular}
\end{center}
    Suppose that $a,b\ne 0$, and consider the standard choice of simple roots, namely $\Sigma=\{\varepsilon_1-\varepsilon_2,\varepsilon_2-\delta_1,\delta_1-\delta_2,\delta_2\}$ . Then
    \begin{align*}
        \Delta^+_{ev}&=\{\varepsilon_1\pm\varepsilon_2,\delta_1\pm\delta_2,2\delta_1, 2\delta_2\}\\
        \Delta^+_{sing} &=\{ \varepsilon_i\pm\delta_j \mid i,j=1,2\}\\
        \Delta^+_{reg}&=\{\varepsilon_1\pm\varepsilon_2,\varepsilon_1,\varepsilon_2,
        \delta_1\pm\delta_2,\delta_1,\delta_2,2\delta_1,2\delta_2\}\\
        \Delta^+_{reg,odd}&=\{\varepsilon_1,\varepsilon_2,\delta_1,\delta_2\}\\
        \Delta_{\Sigma,ad}&=\{\varepsilon_1-\varepsilon_2,\delta_1-\delta_2,\varepsilon_2,\delta_2\}\cup\Delta_{sing}\\
        \Pi&=\{\varepsilon_1-\varepsilon_2,\varepsilon_2,\delta_1-\delta_2,\delta_2\}\\
        W'&=\langle r_{\varepsilon_1-\varepsilon_2},r_{\delta_1-\delta_2}\rangle.\\
    \end{align*}
    Note that $\Delta^+_{reg,odd}$ is a $W'$-invariant set and that $\rho_\Sigma-\rho_{\Sigma}^{pure}=\frac{a}{2}(\varepsilon_1+\varepsilon_2)+2b(\delta_1+\delta_2)$, which is also $W'$-invariant. 
\end{example}

\subsection{Dot action}\label{section actions} Given $w\in W$ and $f\in S(\a)$,  we define the `dot' action, or $\rho$-shifted action, by:
\[
(w_{\smallbullet}f)(\lambda):=f(w(\lambda+\rho_{\Sigma})-\rho_{\Sigma}).
\]
It is important to note that the above dot action depends on $\Sigma$.  We will make this dependence explicit whenever it is not clear from context.  We will say $f\in S(\a)$ is $W_{\smallbullet}$-invariant (resp.~$W'_{\smallbullet}$-invariant) if $w_{\smallbullet}f=f$ for all $w\in W$ (resp.~$w\in W'$).

For $\lambda\in\a^*$, write $t_{\lambda}:\a^*\to\a^*$ for the translation morphism $t_{\lambda}(\mu)=\mu+\lambda$.  This induces an algebra automorphism $t_{\lambda}^*:S(\a)\to S(\a)$ given by 
\[
t_{\lambda}^*(f)(\mu)=f(t_{\lambda}^{-1}(\mu))=f(\mu-\lambda).
\]

\begin{lemma}\label{lemma invce under translation}
    Let $\Sigma\sub\Delta$ be a base, choose $\beta\in\Delta_{reg}$, and let $\alpha\in\Sigma$ be a singular root.  If $f\in S(\a)$ is invariant under the $\rho_{r_{\alpha}\Sigma}$-shifted action of $r_{\beta}$, then $t_{m_{\alpha,1}\alpha}^*(f)$ is invariant under the $\rho_{\Sigma}$-shifted action of $r_{\beta}$. 
\end{lemma}
\begin{proof}
By assumption, 
\[
f(r_{\beta}(\lambda+\rho_{r_{\alpha}\Sigma})-\rho_{r_{\alpha}\Sigma})=f(\lambda).
\]
It follows that:
\begin{eqnarray*}
t_{m_{\alpha,1}\alpha}^*(f)(r_{\beta}(\lambda+\rho_{\Sigma})-\rho_{\Sigma})& = &f(r_{\beta}((\lambda-m_{\alpha,1}\alpha)+\rho_{\Sigma}+m_{\alpha,1}\alpha)-\rho_{\Sigma}-m_{\alpha,1}\alpha)\\
& = & f(r_{\beta}(t_{-m_{\alpha,1}\alpha}(\lambda)+\rho_{r_{\alpha}\Sigma})-\rho_{r_{\alpha}\Sigma})\\
& = & t_{m_{\alpha,1}\alpha}^*(f)(\lambda).
\end{eqnarray*}
\end{proof}

\subsection{Interlacing pairs}
We now make an aside about interlacing pairs.  Let $\k'=\k_{\ol{0}}\oplus\p_{\ol{1}}$; then $(\g,\k')$ forms a supersymmetric pair arising from the involution $\theta\circ\delta$, where $\delta(v)=(-1)^{\ol{v}}v$.

\begin{definition}
    We say that $(\g,\k)$ is interlacing if there exists $a\in\a$ such that $Ad(\exp(a))(\k)=\k'$.
\end{definition}

We see that if $(\g,\k)$ is interlacing then so is $(\g(\alpha),\k(\alpha))$ for all $\alpha\in\Delta$.   For the reader's benefit, we include here a table of all (indecomposable) interlacing pairs.



\begin{center}
    \begin{tabular}{|c|c|}
    \hline 
        $\g$                 & $\k$      \\ \hline\hline
        $\stackrel{\circ}{\g}\times\stackrel{\circ}{\g}$ & $\stackrel{\circ}{\g}$ simple Kac--Moody  \\
        \hline 
        $\g\l(2m|2n+b)$       & $\g\l(m|n)\times\g\l(m|n+b)$, $b\ge 0$    \\ 
        \hline
        $\o\s\p(2m|4n+2b)$    & $\o\s\p(m|2n)\times\o\s\p(m|2n+2b)$, $b\ge 0$        \\
        \hline
        $\o\s\p(2m+a|4n), m\geq 1$, $n\geq 0$       & $\o\s\p(m|2n)\times\o\s\p(m+a|2n)$, $a\ge 0$         \\ \hline
        $\g\l(m|2n)$          & $\o\s\p(m|2n)$       \\ \hline
        $\mathfrak{d}(1,2;\alpha)$ & $\o\s\p(2|2)\times\s\o(2)$\\
        \hline
          $\mathfrak{ab}(1|3)$ & $\g\o\s\p(2|4)$ \\
        \hline
          $\mathfrak{ab}(1|3)$ & $\s\l(1|4)$ \\
          \hline
            $\a\b(1|3)$ & $\mathfrak{d}(1,2;2)\times\s\l(2)$ \\
        \hline
    \end{tabular}
    \end{center}

    



\begin{lemma}\label{lemma interlacing pure}
    The following are equivalent:
    \begin{enumerate}
        \item  $(\g,\k)$ is interlacing.
        \item Both $(\g,\k)$ and $(\g,\k')$ admit Iwasawa decompositions.
        \item $\c(\a)_{\ol{1}}=0$.
        \item $\Delta\sub\a^*$ is pure.
        \item $\rho_{\Sigma}=\rho_{\Sigma}^{\text{pure}}$ for any $\Sigma\sub\Delta$.
    \end{enumerate}
\end{lemma}

\begin{proof}
    (1)$\iff$(2) is proven in \cite[Sec.~4]{Sh2}, and (3)$\Rightarrow$(2) follows from \cite[Thm.~5.3]{Sh4}.  (1)$\Rightarrow$(5) can be checked on $(\g,\k)$ of rank one, where it is easy to check it directly (see Lemma \ref{lemma classification rank one}).
    
    For (5)$\Rightarrow$(4) observe that :
    \[
    \rho_{\Sigma}-\rho_{\Sigma}^{\text{pure}}=\frac{1}{2}\sum\limits_{\alpha\in\Delta_{reg,odd}^+}m_{\alpha,0}\alpha,
    \]
    and this quantity is 0 if and only if $m_{\alpha,0}=0$ for all $\alpha\in\Delta_{reg,odd}^+$.
    
    Finally, for (4)$\Rightarrow$(3), suppose for a contradiction that $\c(\a)_{\ol{1}}\neq0$. Then $\c(\a)_{\ol{1}}$ acts trivially on $\g_{\alpha}$ for all $\alpha\in\Delta$, because all $\g_{\alpha}$ are either purely even or purely odd.  This implies that $\c(\a)_{\ol{1}}$ commutes with the subalgebra generated by all root spaces $\g_{\alpha}$ for $\alpha\in\Delta$.  However, this subalgebra is an ideal, so $\c(\a)_{\ol{1}}$ commutes with this ideal.  Since we assume that $\g$ is indecomposable, the ideal is either everything, the center (which is purely even), or codimension 1.  Clearly none of these are possible, so we obtain a contradiction.

\end{proof}

\section{Invariant polynomials}
\label{Sec: invariant polynomials}

In this section we introduce the different subspaces of the ring of polynomials on $\a$ that we will use in the following sections. Fix throughout a base $\Sigma\sub\Delta$.   

\subsection{The polynomial $T_{\Sigma}$}  For $\alpha\in\Delta_{reg}$, recall that $n_\alpha=m_{\alpha,1}/2$, and define
\[
T_{\alpha,\Sigma}:=\prod\limits_{s=0}^{n_{\alpha}-1}(h_{\alpha}+(\rho_{\Sigma}^{\text{pure}}+(1-n_{\alpha}+2s)\alpha,\alpha)).
\]
For $\alpha\in\Delta_{iso}$, define
\[
T_{\alpha,\Sigma}:=h_{\alpha}+(\rho_{\Sigma},\alpha).
\]
Observe that $T_{\alpha}$ is relatively prime to $T_{\beta}$ for distinct $\alpha,\beta\in\Delta^+$.
\begin{lemma}\label{lemma Talpha under transl}
    Let $\beta\in\Delta$ be either regular or isotropic, and suppose that $\alpha\in\Sigma$ is singular.  Then
    \[
    T_{\beta,\Sigma}=t_{m_{\alpha,1}\alpha}^*(T_{\beta,r_{\alpha}\Sigma}).
    \]
\end{lemma}

\begin{proof}
Indeed, $t_{m_{\alpha,1}\alpha}^*(h_{\beta})=h_{\beta}-(m_{\alpha,1}\alpha,\beta)$.  Thus
\[
t_{m_{\alpha,1}\alpha}^*(h_{\alpha}+(\rho_{r_{\alpha}\Sigma}^{\text{pure}}+r\alpha,\alpha))=(h_{\alpha}+(\rho_{r_{\alpha}\Sigma}^{\text{pure}}-m_{\alpha,1}\alpha+r\alpha,\alpha)).
\]
We may now apply Lemma \ref{lemma pure rho}.
\end{proof}

Define a polynomial $T_{\Sigma}\in S(\a)$ by
\[
T_{\Sigma}:=T_{reg}T_{iso}, 
\]
where 
\[
T_{reg}:=\prod\limits_{\alpha\in\Delta_{reg}^+}T_{\alpha,\Sigma}, \ \ \ T_{iso}:=\prod\limits_{\alpha\in\Delta_{iso}^+}T_{\alpha,\Sigma},
\]
Compare $T_{\Sigma}$ with the element $t$ introduced in \cite[\S 4]{G1}.

\begin{lemma}\label{lemma invce T}
Given the $\rho$-shifted action of $W$ on $S(\a)$,
    \begin{enumerate}
    \item the element $T_{iso}$ is $W_{\smallbullet}$-invariant,
        \item the element $T_{\Sigma}$ is $W_{\smallbullet}'$-invariant, and
        \item if $\alpha$ is a regular odd root, then $T_{\Sigma}/T_{\alpha}$ is $(r_{\alpha})_{\smallbullet}$-invariant.
    \end{enumerate}
\end{lemma}

\begin{proof}
   For (1) it is equivalent to show that $\prod\limits_{\alpha\in\Delta_{iso}^+}h_{\alpha}$ is $W$-invariant with respect to the non-shifted action.  However, this element is independent (up to a sign) of $\Sigma$, and thus we only need to show it is invariant under $r_{\alpha}$ for $\alpha$ simple, when this statement is obvious (see also \cite[Lem.~4.2.1]{G1}). 

If $\alpha$ is regular and odd, then it is orthogonal to all roots in $\Delta_{reg,odd}^+\setminus\{\alpha\}$ by Corollary \ref{cor reg odd roots ortho}, and thus fixes every factor of $T_{reg}/T_{\alpha}$.  In combination with (1), this proves (3). 

 It remains to show that $T_{reg}$ is $W_{\smallbullet}'$-invariant.  For this, we exactly want to show that $T_{reg}$ is $(r_{\alpha})_{\smallbullet}$-invariant, where $\alpha$ runs over the roots of $\Pi$ that are not odd.    After translating by $-\rho_{\Sigma}$, our problem is equivalent to showing that $r_{\alpha}$ fixes
 \[
T_{\Sigma,reg}':=\prod\limits_{\beta\in\Delta_{reg}^+}\prod\limits_{s=0}^{n_{\beta}-1}(h_{\beta}+(\rho_{\Sigma}^{\text{pure}}-\rho_{\Sigma}+(1-n_{\beta}+2s)\beta,\beta)),
 \]

 Note that
  \[
  \rho_{\Sigma}^{\text{pure}}-\rho_{\Sigma}=-\frac{1}{2}\sum\limits_{\delta\in\Delta_{reg,odd}^+}m_{\delta,0}\delta.
  \]
  However, observe that if $\gamma\in\Sigma$ is singular, then $\rho_{\Sigma}^{\text{pure}}-\rho_{\Sigma}=\rho_{r_{\gamma}\Sigma}^{\text{pure}}-\rho_{r_{\gamma}\Sigma}$.  Thus $T'_{\Sigma,reg}=T'_{r_{\gamma}\Sigma,reg}$.  Using Lemma \ref{lemma conj props}, we may therefore assume that $\alpha\in\Sigma$, i.e.~$\alpha$ is simple. 
  
  If $\alpha$ is simple, $r_{\alpha}(\Delta_{reg,odd}^+)\sub\Delta_{reg,odd}^+$.  Further, since the multiplicities $m_{\beta,0}$ are $W$-invariant, $r_{\alpha}$ will fix $\rho_{\Sigma}^{\text{pure}}-\rho$.  From this, invariance follows.
\end{proof}

\begin{example} \label{ex:non interlacing non invariance}
    In the case $\g=\o\s\p(4+a|8+2b)$, $\k=\o\s\p(2|4)\times\o\s\p(2+a|4+2b)$, described in Example \ref{osp example} for the standard choice of $\Sigma$, with $\rho=\rho_{\Sigma}$, we have that
    $$T_{iso}(\lambda-\rho)=\prod\limits_{i,j=1}^2(\lambda,\varepsilon_i+\delta_j)(\lambda,\varepsilon_i-\delta_j).$$
    If $a,b\ne 0$ then $T_{reg}=T_{\varepsilon_1} T_{\varepsilon_2}T_{\delta_1}T_{\delta_2}$ where 
    \begin{align*}
        T_{\varepsilon_i}(\lambda-\rho)&=\prod_{s=0}^{b-1} \left((\lambda,\varepsilon_i)-\frac{a}{2}+1-b+2s\right),\\
        T_{\delta_i}(\lambda-\rho)&=\prod_{s=0}^{a-1} \left((\lambda,\delta_i)+\frac{a-1}{2}+b-s\right).
    \end{align*}
    We see that $T_{reg}(\lambda-\rho)$ is invariant under $W'$, namely permutations of $\{\varepsilon_1,\varepsilon_2\}$ and $\{\delta_1,\delta_2\}$ whereas $T_{iso}(\lambda-\rho)$ is invariant under all signed permutations. If $a$ or $b$ equals to zero then $(\g,\k)$ is interlacing and $T_{reg}=1$.
\end{example}

\subsection{Invariance conditions }\label{sec: invc conditions}

\begin{definition}
   Define $J_{\a}=J_{\a,\Sigma}$ to be the subalgebra of $S(\a)$ consisting of polynomials $f\in S(\a)^{W_{\smallbullet}}$ such that for all singular roots $\alpha$ we have
        \[
        f(\lambda+r\alpha)=f(\lambda-r\alpha), 
        \]
        for $1\leq r\leq n_{\alpha}$ and $(\lambda+\rho_{\Sigma},\alpha)=0$.
\end{definition}

\begin{example}\label{example J osp(2|2)}
    Consider the restricted root system of $(\o\s\p(2|2n),\o\s\p(1|2n))$, which is given by $\Delta=\{\pm\alpha\}\sub\mathbb{C}=\a^*$, where $\alpha$ is purely odd.   Let $h\in\a$ be such that $h(\alpha)=1$, and set $\Sigma=\{\alpha\}$. Observe that $W=W'=\{1\}$ in this case. Then $n_{\alpha}=n$, and we obtain that $J_{\a}$ consists of $f\in\C[t]$ such that
    \[
    f(n+r)=f(n-r) \ \text{ for } \ 1\leq r\leq n.
    \]
    In this case, $J_{\a}$ coincides with the subalgebra of $\C[t]$ generated by $t(t-2n)$ and \linebreak $t(t-1)(t-2)\cdots(t-2n)$.
\end{example}

\begin{definition} \label{def: M_a}
    Define $M_{\a}:=M_{\a,\Sigma}\sub S(\a)$ by those $f\in S(\a)^{W_{\smallbullet}}T_{\Sigma}$ such that for all singular, non-isotropic roots $\alpha$ lying in $W'(\Delta_{\Sigma,ad})$, we have
       \[
       f(\lambda+r\alpha)=(-1)^{r}f(\lambda-r\alpha),
       \]
        for $1\leq r\leq n_{\alpha}$ and $(\lambda+\rho_{\Sigma},\alpha)=0$.
 \end{definition}

Notice that $M_{\a}\sub S(\a)^{W_{\smallbullet}'}$ by Lemma \ref{lemma invce T}. 

\begin{lemma}\label{lemma interlacing invce conditions}
If $\Delta$ is pure (i.e.~$(\g,\k)$ is interlacing), then $M_{\a}$ has the following equivalent description.  It is those $f\in S(\a)^{W_{\smallbullet}'}$ such that
	\begin{enumerate}
       \item for $\alpha$ singular, we have
		\[
		f(\lambda+r\alpha)=(-1)^{r}f(\lambda-r\alpha),
		\]
		for all $(\lambda+\rho_{\Sigma},\alpha)=0$ and $1\leq r\leq n_{\alpha}$.
		\item for $\alpha\in\Delta_{reg,odd}$ we have  
        \[
        (r_{\alpha})_{\smallbullet}(f)=(-1)^{n_{\alpha}}f,
        \]
        and 
		\[
		f(\lambda+r\alpha)=(-1)^{r}f(\lambda-r\alpha),
		\]
        for $1\leq r\leq n_{\alpha}$ and $(\lambda+\rho_{\Sigma},\alpha)=0$.
	\end{enumerate}
\end{lemma}
\begin{proof}
    The first condition is common to both descriptions.  We see that the condition (2) in this Lemma implies that $f$ must be divisible by $T_{\Sigma}$ since $\rho_{\Sigma}=\rho_{\Sigma}^{\text{pure}}$, and for $\alpha\in\Delta_{reg,odd}^+$, $(r_{\alpha})_{\smallbullet}(T_{\Sigma})=(-1)^{n_{\alpha}}T_{\Sigma}$.  Therefore the statement follows.
\end{proof}

\begin{lemma}
    Given $f\in J_{\a}$ and $g\in M_{\a}$, we have $fg\in M_{\a}$.  In particular, $M_{\a}$ has the natural structure of a $J_{\a}$-module.  If $\Delta$ is pure, then $M_{\a}\cdot M_{\a}\sub J_{\a}$, so that $J_{\a}+M_{\a}$ is a subalgebra of $S(\a)$.
\end{lemma}
\begin{proof}
    The first statement is clear, and the second follows quickly from Lemma \ref{lemma interlacing invce conditions}. 
\end{proof}

\begin{example}
    We return to the restricted root system introduced in Example \ref{example J osp(2|2)}.  In this case $M_{\a}$ consists of polynomials $f\in\C[t]$ such that 
    \[
        f(n_{\alpha}+r)=(-1)^rf(n_{\alpha}-r) \ \text{ for } \ 1\leq r\leq n_{\alpha}.
    \]
    We saw in Example \ref{example J osp(2|2)} that $J_{\a}$ contains the subalgebra $A:=\C[t(t-2n_{\alpha})]$.  Then $M_{\a}$ is the $A$-submodule of $\C[t]$ generated by the polynomials
    \[
    g_1:=(t-1)(t-3)\cdots(t-2n_{\alpha}+1), \ \ \ g_2:=t(t-2)(t-4)\cdots(t-2n_{\alpha}).
    \]
    (See \cite[Thm.~9.1(iv)]{Sh2}.)
    Notice incidentally that $g_1g_2$ is the second generator of $J_{\a}$ as mentioned in Example \ref{example J osp(2|2)}.
\end{example}



\begin{example}
    Let $\g=\g\l(2|2)$, $\k=\o\s\p(2|2)$. Choose $\Delta^+=\{\varepsilon_1-\varepsilon_2,\varepsilon_1-\delta,\varepsilon_2-\delta\}$ where $m_{\varepsilon_i-\delta}=(0|2)$ and $m_{\varepsilon_1-\varepsilon_2}=(1|0)$. Then $\rho=-\frac{1}{2}\varepsilon_1-\frac{3}{2}\varepsilon_2+2\delta$. Note that $(\varepsilon_i,\varepsilon_i)=1=-2(\delta,\delta)$. Write $S(\a)=\C[t_{\epsilon_1},t_{\epsilon_2},t_{\delta}]$, where $t_{\varepsilon_1},t_{\varepsilon_2},t_{\delta}$ is dual to $\varepsilon_1,\varepsilon_2,\delta$.  If we consider the polynomial 
    \[
    f_0=(2t_{\varepsilon_1}+t_{\delta})(2t_{\varepsilon_2}+t_{\delta})+1,
    \]
    then it is straightforward to check that $f_0$ belongs to $t_{\rho}^*(M_{\a,\Sigma})$. Indeed, $f$ is invariant when swapping $\varepsilon_1,\varepsilon_2$, and we have
    \[
    f(\lambda+\varepsilon_1-\delta)=-f(\lambda-\varepsilon_1+\delta)
    \]
    whenever $(\lambda,\varepsilon_1-\delta)=0$. It follows that 
    \[
    f(\lambda)=t_{-\rho}^*(f_0)(\lambda)=f_0(\lambda+\rho)=(2t_{\varepsilon_1}+t_{\delta}+1)(2t_{\varepsilon_2}+t_{\delta}-1)+1
    \]
    lies in $M_{\a,\Sigma}$.  By Theorem \ref{main ghost theorem}, one can show that $f$ is the element of lowest degree in $M_{\a,\Sigma}$ (up to scalar).
    
\end{example}

\subsection{rank-one invariance conditions}
Recall the set $\Delta_{\Sigma,ad}$ of admissible roots introduced in Definition \ref{defn admissible roots}.
\begin{definition}
        For $\alpha\in\Delta_{\Sigma,ad}$, define $J_{\a,\Sigma,\alpha}\sub S(\a)$ to be those $f\in S(\a)$ such that:
        \begin{enumerate}
            \item if $\alpha$ is regular then $(r_{\alpha})_{\smallbullet}f=f$;
            \item if $\alpha$ is singular, then
        \[
        f(\lambda+r\alpha)=f(\lambda-r\alpha), \ \ \ 1\leq r\leq n_{\alpha},
        \]
        whenever $(\lambda+\rho,\alpha)=0$.
        \end{enumerate}
\end{definition}   

\begin{lemma}\label{lemma local desc J_a}
    We have 
    \[
    J_{\a,\Sigma}=\bigcap\limits_{\alpha\in\Delta_{\Sigma,ad}}J_{\a,\Sigma,\alpha}
    \]
\end{lemma}
\begin{proof}
    This easily follows from Lemma \ref{lemma conj props}.
\end{proof}

\begin{definition}
    For $\alpha\in\Delta_{\Sigma,ad}$, define $M_{\a,\Sigma,\alpha}\sub S(\a)$ to be those $f\in S(\a)$ satisfying the condition that:
    \begin{enumerate}
        \item if $\alpha$ is regular, then $f\in T_{\alpha,\Sigma}S(\a)^{(r_{\alpha})_{\smallbullet}}$;
        \item if $\alpha$ is singular, then
       \[
       f(\lambda+r\alpha)=(-1)^{r}f(\lambda-r\alpha)
       \]
       for $1\leq r\leq n_{\alpha}$ and $(\lambda+\rho_{\Sigma},\alpha)=0$.
    \end{enumerate}
\end{definition}
Note that if $\alpha$ is isotropic, then condition (2) is equivalent to asking that $f\in T_{\alpha,\Sigma}S(\a)$.

\begin{lemma}\label{lemma local desc M_a}
\[
M_{\a,\Sigma}=\bigcap\limits_{\alpha\in\Delta_{\Sigma,ad}}M_{\a,\Sigma,\alpha}. 
\]
\end{lemma}
\begin{proof}
    It is clear that $M_{\a,\Sigma}\sub\bigcap\limits_{\alpha\in\Delta_{sing}^+}M_{\a,\Sigma,\alpha}$, so it remains to deal with regular admissible roots, which are exactly the elements of $\Pi$.  For this, let $f\in M_{\a,\Sigma}$, so that we may write $f=gT_{\Sigma}$ such that $g$ is $W_{\smallbullet}$-invariant. 
 By Lemma \ref{lemma invce T}, $T_{\Sigma}$ is $W_{\smallbullet}'$-invariant, so we are left with the case that $\alpha\in\Pi$ is an odd regular root.  In this case we have $f/T_{\alpha}=g (T_{\Sigma}/T_{\alpha})$ which is $(r_{\alpha})_{\smallbullet}$-invariant by (3) of Lemma \ref{lemma invce T}.   

 Suppose conversely that $f\in\bigcap\limits_{\alpha\in\Delta_{\Sigma,ad}^+}M_{\a,\Sigma,\alpha}$.   Then $f$ is divisible by $T_{\alpha}$ for all odd roots $\alpha\in\Delta_{\Sigma,ad}^+$.  Since $f$ is $W_{\smallbullet}'$-invariant by condition (1), by (2) of Lemma \ref{lemma conj props} we obtain that $f$ is divisible by $T_{reg}$.  Further, (3) of Lemma \ref{lemma invce T} implies that $f/T_{reg}$ is $W_{\smallbullet}$-invariant.  On the other hand, $f$ is divisible by $T_{\alpha}$ for every admissible odd isotropic root $\alpha$, hence $f/T_{reg}$ must be divisible by $T_{iso}$ by (3) of Lemma \ref{lemma conj props}.  It follows that $f=T_{\Sigma}g$, and since $T_{iso}$ is $W_{\smallbullet}$-invariant, we obtain that $g$ is $W_{\smallbullet}$-invariant.   Finally, to obtain the invariance conditions for singular, non-isotropic roots $\alpha$ lying in $W'(\Delta_{\Sigma,ad})$, we may simply use the $W_{\smallbullet}'$-invariance of $f$.
\end{proof}

\begin{lemma}\label{lemma trans local invce condition}
     Let $\Sigma\sub\Delta$ be a base, choose a singular simple root $\alpha\in\Sigma$, and let $\beta\in\Delta$ be any root.  Then 
     \[
    J_{\a,\Sigma,\beta}=t_{m_{\alpha,1}\alpha}^*(J_{\a,r_{\alpha}\Sigma,\beta}),
    \]
    and 
    \[
    M_{\a,\Sigma,\beta}=t_{m_{\alpha,1}\alpha}^*(M_{\a,r_{\alpha}\Sigma,\beta}).
    \]
\end{lemma}

\begin{proof}
    If $\beta$ is regular or isotropic, this follows from Lemmas \ref{lemma Talpha under transl} and \ref{lemma invce under translation}.  If $\beta$ is singular non-isotropic, then suppose that $f\in J_{\a,r_{\alpha}\Sigma,\beta}$.  Then for $1\leq s\leq m_{\beta,1}/2$ we have for $(\mu+\rho_{\Sigma},\beta)=0$:
    \[
    t_{m_{\alpha,1}\alpha}^*(f)(\mu+s\beta)=f((\mu-m_{\alpha,1}\alpha)+s\beta).
    \]
    If we write $\lambda=\mu-m_{\alpha,1}\alpha$, then by Lemma \ref{lemma rho refl}, $(\lambda+\rho_{r_{\alpha}\Sigma},\beta)=0$, so that 
    \[
    f(\lambda+s\beta)=f(\lambda-s\beta)=f((\mu-m_{\alpha,1}\alpha)-s\beta)=t_{m_{\alpha,1}\alpha}^*(f)(\mu-s\beta).
    \]
    The proof for $M_{\a,\Sigma,\beta}$ works similarly.
\end{proof}

\section{Harish-Chandra projection and relationship to rank one}
\label{Sec:HC proj and rank one}

\subsection{Supersymmetric spaces}  Let $(\g,\k)$ be a supersymmetric pair as above, arising from an involution $\theta$.  Then let $G$ be a quasireductive supergroup, meaning that $G_0^\circ$ is reductive, with Lie superalgebra $\g$.  We assume that $\theta$ lifts to an involution of $G$, and continue to denote this involution by $\theta$.  Finally set $K\sub G$ to be a subgroup with $(G^\theta)^\circ\sub K\sub G^\theta$.  In particular, $K$ will be a quasireductive supergroup.

We call the homogeneous superspace $G/K$ a \emph{supersymmetric variety} (we refer to \cite{MZ} and \cite{MT} for details on homogeneous superspaces).  Notice that $G/K$ is smooth, affine, and admits an action by $G$ on the left by translation, $a:G\times G/K\to G/K$.  Write $\C[G/K]$ for the algebra of regular functions on $G/K$; then $\C[G/K]$ has the natural structure of a right $G$-module via pullback of functions along the action.  This means that $\C[G/K]$ is a left $\C[G]$-comodule, and the coaction is given by
\[
a^*:\C[G/K]\to \C[G]\otimes\C[G/K].
\]
Given $u\in\g$ viewed as an element of $T_eG$, we obtain a vector field on $G/K$ via the formula:
\[
(u\otimes 1)\circ a^*.
\]
The induced linear map $\g\to\operatorname{Der}(\C[G/K])$ is then an anti-homomorphism of Lie algebras, and thus we obtain an algebra homomorphism
\[
(\UU\g)^{op}\to\operatorname{D}(G/K),
\]
where $D(G/K)$ denotes the algebra of differential operators on $G/K$.  In particular, $\C[G/K]$ is a right $\UU\g$-module.

 \subsection{Vector space of distributions} Continuing with the above discussion, set $\operatorname{Dist}(G/K,eK)$ to be the vector space of distributions supported at $eK$ in $G/K$.  To be precise, if we write $\m_{eK}\sub\C[G/K]$ for the maximal ideal of $eK$ in $\C[G/K]$, then we have:
 \[
 \operatorname{Dist}(G/K,eK)=\bigcup\limits_{n\in\N}\left(\C[G/K]/\m_{eK}^n\right)^*.
 \]
 Then $D(G/K)$ acts on $\Dist(G/K,eK)$ on the right via:
 \[
 D\cdot \gamma=\gamma\circ D.
 \]
 Composing with our morphism $(\UU\g)^{op}\to D(G/K)$, we obtain a left action of $\UU\g$ on $\operatorname{Dist}(G/K,eK)$.

 \begin{lemma}
     We have a $\UU\g$-equivariant isomorphism
     \[
     \operatorname{Dist}(G/K,eK)\cong \UU\g/(\UU\g)\k.
     \]
 \end{lemma}
 \begin{proof}
     Because $G/K$ is an affine homogeneous $G$-supervariety, the morphism
     \[
     \UU\g\to \Dist(G/K,eK), \ \ \ u\mapsto u\cdot\delta_{eK}=\delta_{eK}\circ u
     \]
     is surjective, where $\delta_{eK}$ is evaluation at $eK$.  The kernel of this morphism is clearly $(\UU\g)\k$ by the PBW theorem, so we are done.
 \end{proof}

\begin{remark}
    We make two remarks which will be used in the future:
    \begin{enumerate}
        \item The natural quotient map $\UU\g\to\UU\g/(\UU\g)\k$ is $\k$-equivariant, where $\k$ acts on $\UU\g$ by the adjoint action and on $\UU\g/(\UU\g)\k$ by left multiplication.  
        \item The action of $\k_{\ol{0}}$ on $\UU\g/(\UU\g)\k$ is locally finite.
    \end{enumerate}
\end{remark} 
 
	\subsection{Harish-Chandra projection}\label{section HC}  Choose a base $\Sigma\sub\Delta$ to obtain the nilpotent subalgebra $\n=\n_{\Sigma}\sub\g$ (see Section \ref{section iwasawa}).  Using the decomposition (\ref{decomp of Ug from Iwasawa}), we define the \emph{Harish-Chandra projection}
	\[
	HC:=HC_{\Sigma}:\UU\g/(\UU\g)\k\to S(\a)
	\]
to be given by projecting along $(\n\UU\g+(\UU\g)\k)/(\UU\g)\k$ \emph{and} postcomposing with $\sigma_{\a}$, the antipode on $S(\a)$.  Note that we will write $HC$ in place of $HC_{\Sigma}$ when the choice of base is clear.

\begin{remark}
    The antipode $\sigma_{\a}$ will make our eventual formulae in terms of weights more `recognizable'.  (The `issue' is that our description of distributions is via a left-action, meaning we have a right action of $\UU\g$ on the space of functions, so dominant weights are negative what we are accustomed to).
\end{remark}  
\label{spherical weights}
Let $\Lambda^+\sub\a^*$ denote the weights of $\a$ on $\C[G/K]^{\n}$.  For $\lambda\in\Lambda^+$, write 
\begin{equation}\label{eqn f lambda}
f_{\lambda}\in\C[G/K]
\end{equation}
for the highest weight function of weight $\lambda$ such that $f_{\lambda}(eK)=1$; then for $D\in\Dist(G/K,eK)$ we have, by definition:
	\begin{equation}\label{eqn dist}
	D(f_{\lambda})=HC(D)(-\lambda).
	\end{equation}
Note that $\Lambda^+\sub\a^*$ is Zariski dense by \cite[Cor.~5.10]{Sh5} or \cite[Prop.~6.4.2]{Shthesis}; we will use this later in the paper.

\subsection{Commutative diagram of projections} We have the following commutative diagram, where we denote by $HC_{\alpha}$ the Harish-Chandra projection for the pair $(\g(\alpha),\k(\alpha))$ with respect to the positive system $\{\alpha\}$.

\begin{equation}
\label{eq: comm. square}
\xymatrix{\UU\g\ar[rr] \ar[d] & &\UU\g(\alpha)\ar[d]   \\ \UU\g/(\UU\g)\k \ar[rr]^{\operatorname{res}_{\alpha}} \ar[dr]_{HC} & & \UU\g(\alpha)/(\UU\g(\alpha))\k(\alpha)\ar[dl]^{HC_{\alpha}}\\
& S(\a)
}
\end{equation}
Here, the two vertical arrows are the natural quotients.  The horizontal arrows are projections along $\n^{\alpha}\UU\g+(\UU\g) \r$  and $(\n^{\alpha}\UU\g+(\UU\g)\k)/(\UU\g)\k$, respectively (see (\ref{eqn ug decomp at alpha})). All maps in this top square are equivariant with respect to the $\k(\alpha)$-action.  Further, the map $\operatorname{res}_{\alpha}$ is equivariant with respect to the action of $\g(\alpha)$ on the left. 

\begin{cor} \label{cor:HC_alpha}
    For $D\in\UU\g/\UU\g\k$, we have 
    \[
    HC(D)=HC_{\alpha}(\operatorname{res}_{\alpha}(D)).
    \]
\end{cor}

\subsection{List of rank-one pairs}  

We will need the classification of rank-one supersymmetric pairs.  We refer to \cite{S} and the appendix of \cite{Sh4}.

\begin{lemma}\label{lemma classification rank one}
	Suppose that $(\g,\k)$ is rank one so that $\Sigma=\{\alpha\}$.  Then up to split factors fixed by $\theta$, $\g=\g'\times \a'$ where $\theta$ acts on the abelian subalgebra $\a'$ by $(-1)$, and we have the following possibilities for $(\g',\k)$:
 
	\begin{enumerate}
        
		\item $\alpha$ is regular: 

\hfill\break

\begin{center}
    \begin{tabular}{c|c||c|c}
        $\g'$                 & $\k$      &$m_\alpha$   &$m_{2\alpha}$\\ \hline\hline
        $\o\s\p(m|2n)$,    $m\geq 3$, $n\geq 0$       & $\o\s\p(m-1|2n)$       &$(m-2|2n)$   & $(0|0)$ \\ \hline
        $\o\s\p(m|2n)$,    $m\geq 0$, $n\geq 2$       & $\o\s\p(m|2n-2)\times\s\p(2)$       &$(4n-8|2m)$   & $(3|0)$ \\ \hline
        $\g\l(m|n)$,    $m\geq 2$, $n\geq 0$       & $\g\l(m-1|n)\times\g\l(1)$       &$(2m-4|2n)$   & $(1|0)$ \\ \hline
        $\s\l(2)\times\s\l(2)$          & $\s\l(2)$       &$(2|0)$   & $(0|0)$ \\ \hline
        $\o\s\p(1|2)\times\o\s\p(1|2)$           & $\o\s\p(1|2)$       &$(0|2)$   & $(2|0)$ \\ \hline
        $\mathfrak{f}(4)$ & $\s\o(4)$ & $(8|0)$ & $(7|0)$\\
        \hline

    \end{tabular}
\end{center}
\hfill\break

\item  If $\alpha$ is singular
  \hfill\break

\begin{center}
    \begin{tabular}{c|c||c|c}
        $\g'$                 & $\k$      &$m_\alpha$   &$m_{2\alpha}$\\ \hline\hline
        $\o\s\p(2|2n)$,    $n\geq 1$        & $\o\s\p(1|2n)$       &$(0|2n)$   & $(0|0)$ \\ \hline
        $\g\l(1|1)\times\g\l(1|1)$       & $\g\l(1|1)$       &$(0|2)$   & $(0|0)$ \\ \hline

    \end{tabular}
\end{center}
\hfill\break

\color{yellow}
\end{enumerate}
\end{lemma}

\begin{remark}
    Note that the symmetric pair $(\mathfrak{f}(4),\s\o(4))$ does not arise as a rank-one subpair if $\g$ is  indecomposable and $\g_{\ol{1}}\neq0$, meaning that it will not be relevant for our work.  Nevertheless, we include it for completeness.
\end{remark}

\section{Invariant distributions} 
\label{Sec: invariant distributions}
 In this section we describe the ring of $\k$-invariant distributions and prove Theorem \ref{main thm}. To establish bijectivity, we describe the associated graded ring and the relation to the Chevalley restriction map. Additionally, we discuss the relation to the center of $\UU\g$ and its projection to $S(\a)$. Theorem \ref{main thm} is proven in two parts: first we show that the image of $HC$ admits the invariant conditions outlined in Section \ref{sec: invc conditions}. Then we show that $HC$ is bijective in Section \ref{sec: bijectivity inv dist}.

 When not stated, we assume there is a chosen simple root system $\Sigma\sub\Delta$.
	
	\subsection{The ring of $\k$-invariant distributions} Recall that the space of invariant distributions $\ZZ_{(\g,\k)}:=(\UU\g/(\UU\g)\k)^{\k}$ admits the natural structure of an algebra under multiplication induced from $\UU\g$. On the other hand, for a global supersymmetric space $G/K$, write $D^G(G/K)$ for the algebra of invariant differential operators on $G/K$.  The following is standard:
 \begin{lemma}
     We have an isomorphism of algebras $\ZZ_{(\g,\k)}\cong D^G(G/K)$ given by
     \[
     \gamma\mapsto \id\otimes \gamma\circ a^*,
     \]
     with inverse
     \[
    D\mapsto \delta_{eK}\circ D.
     \]
 \end{lemma}


	\subsection{Relationship with the Chevalley restriction map and injectivity of $HC$}
 \label{sec: Chevalley map and injectivity}
    Write $\lambda:S(\p)\to\UU\g$ for the symmetrization map, which we note is $\k$-equivariant under the adjoint action.  The following is clear.
\begin{lemma}
    We have a $\k$-equivariant decomposition $\UU\g=\lambda(S(\p))\oplus\UU\g\k$.  In particular, the composition of $\lambda$ with the natural projection $\UU\g\to\UU\g/\UU\g\k$ is an isomorphism giving the following diagram:
    \[
    \xymatrix{S(\p)^{\k} \ar[r]^{\lambda} \ar[dr]_{\sim} & (\UU\g)^{\k} \ar[d]^{\pi} \\ & \ZZ_{(\g,\k)}}
    \]
    Here the diagonal map is an isomorphism of filtered vector spaces.
\end{lemma}

	By the above, we have an isomorphism of vector spaces $\lambda:S(\p)^{\k}\to \ZZ_{(\g,\k)}$; on the other hand, we have a natural map $CR^*:S(\p)^{\k}\to S(\a)$ given by projection along $\p\cap(\k+\n)\sub\p$.  Notice that our chosen invariant form induces natural isomorphisms $\p\cong \p^*,\a\cong\a^*$, and under these isomorphisms we have a commutative square:
 \[
 \xymatrix{S(\p)^{\k}\ar[d]^{CR^*}\ar[r]^{\sim} & S(\p^*)^{\k} \ar[d]^{CR} \\
 S(\a)\ar[r]^{\sim} & S(\a^*)}
 \]
 where $CR$, the Chevalley restriction map, is the map induced by the projection $\p^*\to\a^*$.

 For a polynomial $f\in S(\a)$, let us write $T(f)$ for the leading term term of $f$.  If $f=0$, set $T(f)=0$.

 \begin{lemma}\label{assoc_graded}
     Let $p\in S(\p)^{\k}$ be homogeneous.  Then
     \[
     CR^*\circ\sigma(p)=T\circ HC\circ \pi\circ\lambda(p),
     \]
     where $\sigma$ is the antipode on $S(\a)$.  Stated otherwise, the following diagram commutes:
     \[
     \xymatrix{
    S(\a) \ar[d]_{\sigma} & S(\p)^{\k}\ar[l]_{CR^*} \ar[r]^{\lambda} & (\UU\g)^{\k}\ar[d]^{\pi} \\
    S(\a) & S(\a)\ar[l]_{T} & \ZZ_{(\g,\k)}\ar[l]_{HC}
     }
     \]
 \end{lemma}

 \begin{proof}
     This is straightforward from our definitions.
 \end{proof}

    The image
$CR^*(S(\p)^{\k})$ in $S(\a)$ was described in \cite{A2} and \cite{RSS}. We use this description to show that $HC$ is surjective in \cref{sec: bijectivity inv dist}. We note that this map does not depend on the choice of $\Sigma$.
We conclude the first part of Theorem \ref{main thm}, namely the injectivity of $HC$ along with a few useful results.

 \begin{cor}\label{cor assoc graded}
\begin{enumerate}
    \item We have $\operatorname{gr}HC=\sigma\circ CR^*$, so that the functor $\operatorname{gr}$ has the following effect:
    	\begin{center}
			\begin{tikzpicture}[decoration=snake]
				\node at (-2,1.5) (Z g k) {$\ZZ_{(\g,\k)}$};
				\node at (2,1.5) (Sa up) {$S(\a)$};
                \node at (-2,-1.5) (S g g) {$S(\p)^{\k}$};
                \node at (2,-1.5) (Sa down) {$S(\a)$};
                
                \draw[->] (Z g k) -- (Sa up) node[above, midway] {$HC$};
                \draw [->] (S g g) -- (Sa down) node[above, midway] {$\sigma\circ CR^*$};
				
				
				\draw[->,decorate] (0,1) -- (0,-0.5);
				
				\node at (0.4,0.2) (gr) {gr};
				
			\end{tikzpicture}
		\end{center}     
    \item $HC:\ZZ_{(\g,\k)}\to S(\a)$ is injective;
    \item Consider the filtration on $HC(\ZZ_{(\g,\k)})$ induced by the natural one on $\ZZ_{(\g,\k)}$ coming from distributions.  Then we have:
    \[
    \operatorname{gr}HC(\ZZ_{(\g,\k)})\cong CR^*(S(\p)^{\k}).
    \]
\end{enumerate}
 \end{cor}

 \begin{proof}
     Part (1) follows from Lemma \ref{assoc_graded}, and (2) follows from the injectivity of $CR$ (Prop.~4.2 of \cite{RSS}, see also Prop.~1 of \cite{Sergeev}) and Lemma \ref{lemma assoc gr reflects}.  Part (3) follows from (1) and (2).
 \end{proof}

	

\subsection{The center of $\UU\g$}

We recall the Harish-Chandra theorem for the center $Z(\g)$ of $\UU\g$. Let $\g=\u^-\oplus\h\oplus\u^+$ be a triangular decomposition of $\g$. We assume that $\u^+$ contains $\n$.
We define $\widetilde{HC}:\UU \g\rightarrow S(\h)$ to be the `usual' Harish-Chandra projection, given by projecting along $\u^-\UU\g+(\UU\g)\u^+$.

\begin{lemma}
    The following diagram is commutative:
\[
\xymatrix{Z(\g)\ar[rr]^{\operatorname{proj}} \ar[d]_{\widetilde{HC}}  & & \ZZ_{(\g,\k)}  \ar[d]_{HC}  \\ 
S(\h) \ar[rr]^{\operatorname{proj}_{\a}} & & J_{\a}
}
\]
Here, $\operatorname{proj}$ is the composition of $Z(\g)\sub \UU\g\to \UU\g/(\UU\g)\k$ which lands in $\ZZ_{(\g,\k)}$, and $\operatorname{proj}_{\a}$ is the quotient arising from the decomposition $\h=\a\oplus\mathfrak{t}$.
\end{lemma}
\begin{proof}
    We prove that $HC(\operatorname{proj}(z))=\operatorname{proj}_{\a}(\widetilde{HC}(z))$ for $z\in Z(\g)$ by showing they take the same values on $\Lambda^+\sub\a^*$, which we noted in Section \ref{spherical weights} is dense in $\a^*$.  
    For $z\in Z(\g)$, let $\Theta$ denote the differential operator on $G/K$ induced by $z$.  
    On the one hand, we have by (\ref{eqn dist}):
    \[
    (\operatorname{ev}_{eK}\Theta)f_{\lambda}=HC(\operatorname{proj}z)(-\lambda).
    \]
    On the other hand, we have
    \[
(\operatorname{ev}_{eK}\Theta)f_{\lambda}=(\Theta f_{\lambda})(eK)=\widetilde{HC}(z)(-\lambda)f_{\lambda}(eK)=\widetilde{HC}(z)(-\lambda)
    \]
    Note that $\Theta f_{\lambda}=\widetilde{HC}(z)(-\lambda)f_{\lambda}$ because we are considering a right action.  The result is now clear.
\end{proof}
\begin{remark}
    One can prove the lemma by direct computation in $\UU\g$, as follows. Since the maps in the diagram are projection maps, it is enough to show that the kernels are equal when intersected with $Z(\g)$.

Let $f\in Z(\g)^\g\cap(\u^-\UU\g+(\UU\g)\u^+)$ and write $f$ as a linear combination of monomials with PBW basis in order that corresponds to the decomposition $$\g=(\u^-\cap\n^-)\oplus(\u^-\cap\k)\oplus\mathfrak t\oplus \a\oplus (\u^+\cap\k)\oplus (\u^+\cap\n)$$
where $\n^-=\bigoplus_{\alpha\in-\Delta^+}\g_\alpha$.
If the part from $\u^+ \cap \n$ is nonzero, the monomial is in $(\UU\g)\n$, and thus in the kernel of $HC$. If the part from $\u^+ \cap \n$ is zero, then the part from $\u^- \cap \n^-$ must be zero as well because the monomial has zero $\h$-weight and, specifically, zero $\a$-weight. Consequently, the monomial is in $\k \UU\g$, and therefore in the kernel of $HC$.

Suppose $f\in Z(\g)^\g\cap(\k\UU\g+(\UU\g)\n)$ and write $f$ as a linear combination of monomials with PBW basis in order that corresponds to the decomposition 
$$\g=(\m\cap\u^-)\oplus\mathfrak t\oplus\k'\oplus(\m\cap\u^+)\oplus\a\oplus\n$$
where $\k'=\operatorname{span}\{\theta e_\alpha+e_\alpha\mid\alpha\in\tilde\Delta^+\}$.
If the part in $\n$ or in $\m\cap\u^-$ is nonzero, the monomial is in $\ker\widetilde{HC}$. 
We note that $\m\cap \u^+$ commutes with $\a$ so if the part in $\m\cap \u^+$ is nonzero, the monomial is also in $\ker\widetilde{HC}$. 
We are left with monomials generated by elements in $\mathfrak t\oplus\k'\oplus\a$. 
Since the monomials of $f$ must be of $\h$-weight zero, the part in $\k'$ must be zero. Indeed, a product of the form $\prod_{\alpha>0}(e_\alpha+\theta e_\alpha)$ will contain $\prod_{\alpha>0}e_\alpha$, which has a positive weight. Hence the monomials of $f$ are generated by $\mathfrak t\oplus \a$ and $\mathfrak t$ is annihilated in $\operatorname{proj}_\a$, completing the proof.

\end{remark}

\begin{cor}\label{cor center surj}
        The associated graded of the map $\operatorname{proj}_{\a}\circ \widetilde{HC}:Z(\g)\to S(\a)$ is given by the natural composition of projections $R:S(\g)^{\g}\to S(\h)\to S(\a)$.  In particular, if $R$ restricts to a surjection $S(\g)^\g\to CR(S(\p)^{\k})$, then $Z(\g)\to HC(\ZZ_{(\g,\k)})$ is also surjective. The cases for which $R$ restricts a surjection are listed in \cite[Thm.~B]{RSS}.
\end{cor}

\subsubsection{Description of $\widetilde{HC}(Z(\g))$} By the previous section, we know that
\[
\operatorname{proj}_{\a}\widetilde{HC}(Z(\g))\sub HC(\ZZ_{(\g,\k)}).
\]
We recall, for the benefit of the reader, the Sergeev--Kac--Gorelik description which describes $\widetilde{HC}(Z(\g))$.  

Let $\tilde \Delta=\tilde \Delta_{\bar 0}\cup\tilde\Delta_{\bar 1}\subset \h^*$ denote the set of roots and let $\tilde\Delta^+$  be the set of positive roots corresponding to $\u^+$. 
Let $\tilde\rho=\frac{1}{2}\left(\sum_{\alpha\in\tilde\Delta_{\bar 0}}\alpha -\sum_{\alpha\in\tilde\Delta_{\bar 1}}\alpha\right)$.  
Let $\widetilde W$ be the Weyl group with respect to $\tilde{\Delta}_{\bar 0}$, and recall the $\tilde{\rho}$-shifted action on $\h^*$: $w_{\smallbullet}\lambda= w(\lambda+\tilde{\rho})-\tilde{\rho}$. 

The Harish-Chandra image of the center of the universal enveloping algebra can be described in the following theorem.
\begin{thm}[\cite{G2,K,Sergeev}]\label{thm center}
The image of $\widetilde{HC}$ applied to $Z(\g)$ is equal to 
$$\left\{
f\in S(\h)^{\widetilde W_{\smallbullet}} \mid f(\lambda+\alpha)=f(\lambda) \text{ for }\lambda\in H_\alpha, \alpha\in\tilde\Delta_{\bar 1}\text{ s.t. } (\alpha,\alpha)=0
\right\},$$
where $H_\alpha:=\{\lambda\in\h^*\mid (\lambda+\tilde\rho,\alpha)=0\}$. Here the action of $\widetilde W$ is the $\tilde\rho$-shifted action.
\end{thm}

\begin{remark}\label{rem:diagonal case}
    Theorem \ref{thm center} is equivalent to the computation of $HC(Z(\g\times\g,\g))$, where we use the standard identification $\UU(\g\times\g)/\UU(\g\times\g)\Delta\g\cong \UU\g$, where $\Delta\g\sub\g\times\g$ is the diagonally embedded copy of $\g$.
\end{remark}


\subsection{A rank-one reduction} Let $\alpha\in\Sigma$.  Using the commutative square in Equation (\ref{eq: comm. square}) and the equivariance of the action of $\k(\alpha)$, we obtain the following commutative triangle:
	\begin{equation}\label{eq comm diag inv dist}
	\xymatrix{\ZZ_{(\g,\k)}\ar[rr] \ar[dr]_{HC} && \ZZ(\g(\alpha),\k(\alpha)) \ar[dl]^{HC_{\alpha}}\\ & S(\a) &}
	\end{equation}	
	where $HC_\alpha$ is the Harish-Chandra map for the rank-one symmetric pair $(\g(\alpha),\k(\alpha))$, defined in Section \ref{subsec:rank one prelim}. In particular, we obtain the following lemma.
    
    \begin{lemma}
        For a simple root $\alpha$, we have $HC(\ZZ_{(\g,\k)})\sub HC_{\alpha}
        (\ZZ(\g(\alpha),\k(\alpha)))$.
    \end{lemma}

 \begin{prop}\label{prop_rank_one}
 Suppose that $(\g,\k)$ is a rank-one symmetric pair with simple root $\alpha$. 
 Then $HC(\ZZ_{(\g,\k)})=J_{\a}$.
		\end{prop}
		\begin{proof}
        We use the classification given in Lemma \ref{lemma classification rank one}.  Without loss of generality, we may assume that $\a'=0$.
        
    		For Case (1), we have $\C[HC(\Omega)]\sub J_{\a}$, where $\Omega$ is the quadratic Casimir element, and $HC(\Omega)$ is a nonzero element of degree 2.  
            On the other hand, it is clear that $CR^*(S(\p)^\k)\sub S(\a)^{r_{\alpha}}$.  Since $\dim\a=1$, $S(\a)^{r_\alpha}$, is the ring of even degree polynomials in one variable and is equal to $\operatorname{gr}\C[HC(\Omega)]$. We can thus conclude by Corollary \ref{cor assoc graded}.
   			
			For Case (2), we use our classification of such pairs given in Lemma \ref{lemma classification rank one}. 
            The case of the diagonal pair $(\g\l(1|1)\times\g\l(1|1),\g\l(1|1))$ follows from Theorem \ref{thm center}, see \cref{rem:diagonal case}.  
            It remains to consider the pair $(\g,\k)=(\o\s\p(2|2n),\o\s\p(1|2n))$.  
            By \cite[Thm.~B]{RSS}, for $(\g,\k)$ of rank one, the map $S(\g)^{\g}\to S(\p)^\k$ is surjective, so we know from Corollary \ref{cor center surj} that $J_{\a}$ is the image of the composition $Z(\g)\xto{\widetilde{HC}} S(\h)\to S(\a)$. 
            On the other hand, \cite[Thm.~A]{RSS} or \cite[Prop. 4.5]{A2} imply that in this case $CR^*(S(\p)^\k)$ is of codimension $n$ in $S(\a)$.  
            Since our proposed description of $J_{\a}$ also has codimension $n$ in $S(\a)$, it only remains to show that the described invariance conditions hold.

            Let $\epsilon\in\a^*$ be the basis element such that $\epsilon(t)=1$, where $t\in\a$ and $\ad(t)$ has eigenvalues $0,\pm1$ on $\g$.  This gives us a $\Z$-grading $\g=\g_{-1}\oplus\g_0\oplus\g_1$, and we choose our positive system so that $\n^+=\g_1$.  Set $K(s)=\Ind_{\g_0\oplus\g_1}^{\g}\C_{s\epsilon}$ to be the Kac-module of highest weight $s\epsilon$. Notice that since $\k+(\g_0+\g_1)=\g$, we have
            \[
            \operatorname{Res}_{\k}K(s)\cong \Ind_{\k_{\ol{0}}}^{\k}\C.
            \]
            This $\k$-module is semisimple and admits a one-dimensional $\k$-fixed subspace, giving us a unique map to our global supersymmetric space:
            \[
            K(s)\to\C[G/K].
            \]
            In particular we have highest weight functions $f_{s}\in\C[G/K]^{\n^+}$ for every $s\in\Z$.  

            In Sec.~9.1 of \cite{Sh2}, it was shown that $f_{n-r}\in\UU\g\cdot f_{n+r}\sub \C[G/K]$ whenever $1\leq r\leq n$.  It follows that any element of $\ZZ_{(\g,\k)}$ must act on $f_{n-r}$ and $f_{n+r}$ by the same scalar.  This implies the invariance condition, using that $\rho=-n\epsilon$.
		\end{proof}


  \begin{remark}
      We note that if $\alpha$ is isotropic and $(\lambda+\rho,\alpha)=0$, then $(\lambda+\alpha+\rho,\alpha)=0$ which implies  that $f(\lambda)=f(\lambda+r\alpha)$ for every $r\in \mathbb{Z}$.  Since $f$ is a polynomial, this forces $f(\lambda+r\alpha)=f(\lambda)$ for all $r\in\C$.
\end{remark}

  \begin{remark}
      
     The invariance for the non-isotropic singular case can be also be deduced explicitly by restricting an element in $HC(Z(\o\s\p(2|2n)))$ to $\a$. Recall that in this case, $\h^*=\operatorname{span}\{\varepsilon,\delta_1,\ldots,\delta_n\}$ and $\a^*=\operatorname{span}\{\varepsilon\}$. Consider the standard choice of simple roots for $\g$. Then 
     $$\tilde \rho=-n\varepsilon+n\delta_1+(n-1)\delta_2+\cdots+\delta_n \text{ and } \rho=-n\varepsilon.$$
     The isotropic roots of $\o\s\p(2|2n)$ are $\{\pm(\varepsilon\pm\delta_i)\mid  i=1,\ldots,n\}$ and so the restricted roots are $\{\pm\varepsilon\}$ with multiplicity $(0|2n)$.
     
     A function $f\in HC(Z(\o\s\p(2|2n)))$ satisfies 
     $$f(\lambda)=f(\lambda+r\alpha)\text{ for any }r\in\mathbb Z ,\lambda\in \tilde H_{\alpha}$$
     where $\tilde H_{\alpha}=\{\lambda\in\h^* \mid(\lambda+\tilde\rho,\alpha)=0 \}$. For $r=1,\ldots,n$, $(n-r)\varepsilon\in \tilde H_{\varepsilon-\delta_{n-r+1}}$ so
     $$f((n-r)\varepsilon)=f(n\varepsilon-r\delta_{n-r+1}).$$
     However $n\varepsilon-r\delta_{n-r+1}\in \tilde H_{\varepsilon+\delta_{n-r+1}}$ so
     $$f(n\varepsilon-r\delta_{n-r+1})=f((n+r)\varepsilon).$$
     We obtain that $f((n-r)\varepsilon)=f((n+r)\varepsilon)$. Since $\lambda\in H_\varepsilon$ if and only if $\lambda=n\varepsilon$, we are done. 
      \end{remark}

      
      \begin{cor}\label{cor invt rank one simple inclusion}
          For any simple root $\alpha\in\Sigma$, we have $HC(\ZZ_{(\g,\k)})\sub J_{\a,\Sigma,\alpha}$.
      \end{cor}
      \begin{proof}
          By Proposition \ref{prop_rank_one}, we have that $HC(\ZZ_{(\g,\k)})\sub J_{\a,\{\alpha\}}$.  However, because $\alpha$ is simple we may use Lemma \ref{lemma local global rho} to obtain that $J_{\a,\{\alpha\}}=J_{\a,\Sigma,\alpha}$.
      \end{proof}
      
	\subsection{Changing positive systems}  Let $\alpha\in\Sigma$, and consider the reflected simple root system $r_{\alpha}\Sigma$ (see Section \ref{sec reflections}).  Recall that $t_{m_{\alpha,1}\alpha}^*$ denotes the translation automorphism of $S(\a)$ in $m_{\alpha,1}\alpha$ (see Section \ref{section actions}). 
	
	\begin{prop}\label{prop: reflections on HC}
    Let $D\in \ZZ_{(\g,\k)}$. One has:
		\begin{enumerate}
			\item if $\alpha$ is a regular root then $HC_{\Sigma}(D)=HC_{r_{\alpha}\Sigma}(D)\circ r_{\alpha}$;
			\item if $\alpha$ is a singular root then we have that $HC_{\Sigma}(D)=t_{m_{\alpha,1}\alpha}^*(HC_{r_{\alpha}\Sigma}(D))$. In particular,
            \[
            HC_{\Sigma}(\ZZ_{(\g,\k)})=t_{m_{\alpha,1}\alpha}^*(HC_{r_{\alpha}\Sigma}(\ZZ_{(\g,\k)}))
            \]
		\end{enumerate}
	\end{prop}

\begin{proof} 
It suffices to prove our formulae for $\lambda\in\Lambda^+$, since the set $\Lambda^+$ is dense in $\a^*$ by \cref{spherical weights}.  For $\lambda\in\Lambda^+$, let $f_{\lambda}$ denote the corresponding eigenfunction on $G/K$ (see (\ref{eqn f lambda})).  Then consider $V=(\UU\g) f_{\lambda}$.  We may ask what is the highest weight of $V$ with respect to $\n_{r_{\alpha}\Sigma}$.  

First suppose that $\alpha$ is regular.  Then if we choose a lift $\widetilde{r_{\alpha}}$ of $r_{\alpha}$ to $K_0^\circ$, then $\widetilde{r_{\alpha}}^*(f_{\lambda})=f_{r_{\alpha}(\lambda)}$ will be a highest weight vector with respect to $r_{\alpha}(\Sigma)$.  Since $D$ is $K_0^{\circ}$ invariant, we have 
    \[
    HC_{\Sigma}(D)(r_{\alpha}(\lambda))=D(f_{r_\alpha(\lambda)})=D(\widetilde{r_{\alpha}}^*(f_{\lambda}))=D(f_{\lambda})=HC_{r_{\alpha}\Sigma}(D)(\lambda).
    \]
This proves (1).  Now we assume that $\alpha$ is singular.  Then by Lemma \ref{lemma classification rank one}, up to split factors fixed by $\theta$, $(\g(\alpha),\k(\alpha))$ is either $(\o\s\p(2|2n),\o\s\p(1|2n))$ or $(\g\l(1|1)\times\g\l(1|1),\g\l(1|1))$, and so we may assume $\g(\alpha)$ is either $\o\s\p(2|2n)$ or $\g\l(1|1)\times\g\l(1|1)$ .  In both cases we obtain a type one Lie superalgebra $\g(\alpha)=\g(\alpha)_{-1}\oplus\g(\alpha)_0\oplus\g(\alpha)_1$, and the distinguished Borel subalgebra $\b_0\oplus\g(\alpha)_1$ will be an Iwasawa Borel.  Thus for $\lambda\in\Lambda^+$ typical, $\UU\g(\alpha)f_{\lambda}$ will be isomorphic to an irreducible Kac-module for $\g(\alpha)$.  In particular, ${\bigwedge}^{top}\g(\alpha)_{-1}f_{\lambda}$ contains the lowest weight vector $g$, and hence is of weight $t_{-m_{\alpha,1}\alpha}(\lambda)$.  Since $\UU\g(\alpha)f_{\lambda}=\UU\k(\alpha)f_{\lambda}$ (by the Iwasawa decomposition), we obtain that 
\[
HC_{\Sigma}(D)(\lambda)=D(f_{\lambda})=D(g)=HC_{r_{\alpha}\Sigma}(D)(t_{-m_{\alpha,1}\alpha}(\lambda)).
\]
This formula now implies (2), and we are done.

\end{proof}

\begin{cor}\label{cor invce for HC}
    We have $HC_{\Sigma}(\ZZ_{(\g,\k)})\sub J_{\Sigma,\a}$. 
\end{cor}

\begin{proof}
We claim that 
\[
HC_{\Sigma}(\ZZ_{(\g,\k)})\sub \bigcap\limits_{\alpha\in\Delta_{\Sigma,ad}^+}J_{\a,\Sigma,\alpha},
\]
which implies the statement by Lemma \ref{lemma local desc J_a}.  

To prove our claim, observe first that by Corollary \ref{cor invt rank one simple inclusion}, $HC_{\Sigma}(\ZZ_{(\g,\k)})\sub J_{\a,\Sigma,\alpha}$ for all simple $\alpha$.  Now suppose that $\beta\in \Sigma':=r_{\alpha_k}\cdots r_{\alpha_1}\Sigma$, where $\alpha_i\in r_{\alpha_{i-1}}\cdots r_{\alpha_1}\Sigma$ is a singular root.  By Corollary \ref{cor invt rank one simple inclusion}, $HC_{\Sigma'}(\ZZ_{(\g,\k)})\sub J_{\a,\Sigma',\beta}$; however, by (2) of Proposition \ref{prop: reflections on HC} this implies that
\[
HC_{\Sigma}(\ZZ_{(\g,\k)})\sub t_{m_{\alpha_k,1}\alpha_k}^*\left(\cdots\left(t_{m_{\alpha_1,1}\alpha_1}^*(J_{\a,\Sigma',\beta})\right)\right)=J_{\a,\Sigma,\beta},
\]
where the final equality follows from Lemma \ref{lemma trans local invce condition}.  This completes the proof.

\end{proof}



\subsection{Surjectivity of the Harish-Chandra map} \label{sec: bijectivity inv dist}

\subsubsection{The Associated-Graded Functor}
Given an $\mathbb{N}$-filtered vector space $V=\bigcup_{i\ge 0}V_i$ where $V_i\subseteq V_{i+1}$, we can associate the space $\gr V=\bigoplus_{i\ge 0} V_{i+1}/V_i$. Given a linear map $f:V\rightarrow W$ between two filtered vector spaces (which preserves the filtration), we can take the corresponding map $\gr f: \gr V\rightarrow \gr W$ defined by $\gr f(v_i)=f(v_i)+W_{i-1}$ for $v_i\in V_i$ and $V_{-1},W_{-1}:=\{0\}$.  The following lemma is well-known.

\begin{lemma}\label{lemma assoc gr reflects}
    Suppose that $f:V\rightarrow W$ is a map of filtered vector spaces. If $\gr f$ is bijective then so is $f$.
\end{lemma}

\subsubsection{Proof of surjectivity}  To prove the bijectivity of $HC$, we use the Chevalley restriction theorem of symmetric superspaces. Recall the dual Chevalley restriction map $CR^*=\operatorname{gr}HC$ from Section \ref{sec: Chevalley map and injectivity}. By \cite{A2, RSS}, the image of $CR^*$  is equal to the set $I_\a\subset S(\a)$ of $W_{\a}$-invariant polynomials such that for all singular roots $\alpha$, we have:
		\[
		(D_{\alpha}^rf)(\lambda)=0 \text{ if } (\lambda,\alpha)=0,
		\]
		\[
		(D_{\alpha}^rf)(\lambda)\in\langle h_\alpha\rangle,
		\]
		where $r=1,3,\dots,m_{\alpha,1}-1$.

\begin{prop}
    One has $\operatorname{gr} J_\a=I_\a$. Therefore, $HC$ is surjective onto $J_\a$.
\end{prop}

We note that the algebra $J_\a$ depends on the choice of $\Sigma$ whereas the algebra $I_\a$ does not.

\begin{proof}
The map $CR:S(\p)^\k\to I_{\a}$ is an isomorphism, so by Lemmas \ref{assoc_graded} and \ref{lemma assoc gr reflects}, the first statement implies the second.  Moreover, it is enough to show that $\operatorname{gr} J_\a\subseteq I_\a$, as follows from the following diagram
\[
\xymatrix{\operatorname{gr}(\ZZ_{(\g,\k)})\ar[rr]^{\operatorname{gr}(\operatorname{HC})} \ar[d]_{\cong}  & & \operatorname{gr}(J_\a) \ar[d]_{\subseteq}  \\ 
S(\p)^\k \ar[rr]^{CR\ \cong} & & I_{\a}
}
\]
Let $f\in J_\a$. Then $f$ is $W_{\smallbullet}$-invariant, and so $\operatorname{gr}f$ is $W$-invariant. It remains to show the condition for singular roots.

Let $\alpha$ be a  singular root. Then 
	\[
	f(\lambda+r\alpha)=f(\lambda-r\alpha),\quad1\le r\le n_{\alpha}=\frac{m_{\alpha,1}}{2}.
	\]
for any $\lambda\in H_{\alpha}:=\left\{ \lambda\in\mathfrak{a}^{*} \mid (\lambda+\rho,\alpha)=0\right\} $.
Let $h_\alpha$ be the dual element of $\alpha$, that is $\lambda(h_\alpha)=(\lambda,\alpha)$ for any $\lambda\in\a^*$.  Denote $m:=\deg f$ and let $\bar{f}$ be the degree $m$-term of $f$. 
We need to show that $D_{\alpha}^r\left(\ol{f}\right)\in\left\langle h_{\alpha}\right\rangle$ for	
$r=1,3,\ldots,m_{\alpha,1}-1$.

Suppose first that $\alpha$ is not isotropic.
 Then $\mathfrak{a}=\ker\alpha\oplus\text{span}\left\{h_\alpha\right\} $.  We need to prove that $\bar f\in\sum_{i\in A}  h_\alpha^i S(\ker\alpha)$ where $A=2\Z_{\ge 0}\cup \Z_{\ge m_{\alpha,1}}$, namely all odd degrees of $h_\alpha$ are greater than $m_{\alpha,1}$. 
Note that $f_\rho(\lambda):=f(\lambda-\rho)$ has the same top degree element as $f$ so it is enough to prove the claim for $f_\rho$.
	Write $f_\rho=\sum_{i}k_{i}h_{\alpha}^{i}$ where $k_{i}\in S(\ker\alpha)$. 
    Note that $k_i(\lambda+r\alpha)=k_i(\lambda)$ for any $r\in\mathbb Z$.  
    Hence
for $r=1,\ldots,n_{\alpha}$, we have
	\[
	\sum_{i}k_{i}r^{i}
        =\sum_{i}k_{i}\left(-r\right)^{i},
	\]
	which is equivalent to 
	\[
	\sum_{ i \text{ odd}} k_{i}r^{i}=r\sum_{i=0}k_{2i+1} (r^2)^i=0 \text{ for }r=1,\ldots,n_{\alpha}.
	\]
It follows that the function $f_\rho':=\sum_{ i \text{ odd}} k_{i}h_\alpha^{i}$ has the form 
$$f_\rho'=Kh_\alpha\left(h_\alpha^2-1^2\right)\left(h_\alpha^2-2^2\right)\cdots\left(h_\alpha^2-\left(\frac{m_{\alpha,1}}{2}\right)^2\right)$$
for some $K\in S(\ker\alpha)$.
 Hence the $m$-degree term of $f_\rho'$ is divisible by $h_\alpha^{m_{\alpha,1}+1}$, which implies our desired claim for $f_{\rho}$.

If $\alpha$ is isotropic, then we have that $f(\lambda+\alpha)=f(\lambda)$ for $\lambda\in H_{\alpha}$. Also, by Lemma \ref{lemma classification rank one}, in this case $m_{\alpha,1}=2$.  This implies that $D_{\alpha}(f)$ vanishes on $H_{\alpha}$, and thus we may write $D_{\alpha}(f)=(h_{\alpha}+(\rho,\alpha))g$ for some $g\in S(\a)$.  This implies that $f=(h_{\alpha}+(\rho,\alpha))\tilde{g}+k$, where $k\in S(\ker\alpha)$ and $\tilde{g}\in S(\a)$ satisfies $D_{\alpha}(\tilde{g})=g$. Hence we see that $D_{\alpha}(\ol{f})=h_{\alpha}\ol{g}$ or is 0, so we are done.
\end{proof}

\section{Ghost Distributions and their Harish-Chandra image}
\label{Sec: Ghost proof}

In this section we study the Harish-Chandra image of ghost distributions on supersymmetric spaces.  We continue to work with a fixed base $\Sigma\sub\Delta$.

\subsection{Ghost distributions}

Let $\k'=\k_{\ol{0}}\oplus\p_{\ol{1}}$.  Note that because $\k_{\ol{0}}$ acts locally finitely on $\UU\g/(\UU\g)\k$, $\k'$ will as well; in fact we have an isomorphism of $\k'$-modules (\cite[Prop.~5.7]{Sh1}):
\[
\UU\g/(\UU\g)\k\cong\Ind_{\k_{\ol{0}}}^{\k'}\UU\g_{\ol{0}}/(\UU\g_{\ol{0}})\k_{\ol{0}}.
\]
Define the ghost distributions of $(\g,\k)$ to be:
\[
\AA_{(\g,\k)}:=\left(\UU\g/(\UU\g)\k\right)^{\k'}.
\]
Note that if we write $v_{\k'}$ for some nonzero element of $(\UU\k'/(\UU\k')\k_{\ol{0}})^{\k'}$, then we have a natural isomorphism of vector spaces (\cite[Prop.~7.1]{Sh1}):
\begin{equation}\label{eqn ghost dist pres}
\left(\UU\g_{\ol{0}}/(\UU\g_{\ol{0}})\k_{\ol{0}}\right)^{\k_{\ol{0}}}\to\AA_{(\g,\k)},  \ \ \ \ z\mapsto z v_{\k'}.
\end{equation}
We consider the Harish-Chandra map $HC:\AA_{(\g,\k)}\to S(\a)$ as described in Section \ref{section HC}.  In \cite[Thm.~1.1]{Sh2} it is shown that (\ref{eqn ghost dist pres}) is an embedding. 

To understand the associated graded of the image, we use \cite[Cor.~3.3]{Sh2} which says the following: there exist nonzero elements $h_1,\dots,h_k\in S(\a)$, with $k=\dim\p_{\ol{1}}/2=\dim\n_{\ol{1}}/2$, such that $HC(D v_{\k'})=\pm h_1\cdots h_kHC(D)+\text{l.o.t.}$, where we use the notation of (\ref{eqn ghost dist pres}).  In fact, it follows from the proof of \cite[Thm. 3.1]{Sh2} that  the associated graded is
\begin{equation}\label{eq:graded dim}
\operatorname{gr}HC(\AA_{(\g,\k)})=
\prod_{\alpha\in\Delta_1^+} h_\alpha^{n_\alpha}\cdot S(\a)^W.
\end{equation} 

The main theorem and central goal of this section is now the following statement.

\begin{thm}
    We have $HC(\AA_{(\g,\k)})=M_{\a,\Sigma}$.
\end{thm}


        

\subsection{Invariance corresponding to simple roots} 

Let $\alpha\in\Sigma$ be a simple root, and consider the commutative diagram  (\ref{eq: comm. square}).  Because $\operatorname{res}_{\alpha}$ is $\g(\alpha)$-equivariant, we may take $\k(\alpha)'$-invariance in the bottom triangle of (\ref{eq: comm. square}) to obtain the following commutative diagram, which is the analogue to (\ref{eq comm diag inv dist}):
\[
\xymatrix{
\AA_{(\g,\k)}\ar[rr]\ar[dr] && \AA_{(\g(\alpha),\k(\alpha))} \ar[dl] \\ & S(\a) & }
\]
We thus obtain that for any simple root $\alpha\in\Sigma$,
\[
HC(\AA_{(\g,\k)})\sub HC(\AA_{(\g(\alpha),\k(\alpha))}).
\]


\begin{prop}\label{prop rank one ghost}
If $(\g,\k)$ is a rank-one supersymmetric pair with chosen simple root $\alpha$ then $HC(\AA_{(\g,\k)})=M_{\a,\{\alpha\}}$.
\end{prop}

\begin{proof}

We will prove this by inspection of all rank-one supersymmetric pairs.  First, suppose that $\Delta\sub\a^*$ is a rank-one root system, $\rho=-a\alpha$ and $\rho^{pure}=-b\alpha$ for $a\in\frac{1}{2}\Z$ and $b\in\mathbb{N}$. Define $h_{\alpha}\in\a$ by $\lambda(h_{\alpha})=(\alpha,\lambda)$, and write $k_{\alpha}:=(\alpha,\alpha)$.

	\renewcommand{\arraystretch}{1.3}	

\begin{table}[h!]

\centering
\begin{tabular}{|c|c|}
\hline
Case & $M_{\a,\{\alpha\}}$ \\
\hline
$\alpha$ isotropic & $h_{\alpha}S(\a)$ \\
\hline
$\alpha$ singular & $\{f\in S(\a):f(a\alpha+r\alpha)=(-1)^{r}f(a\alpha-r\alpha)\text{ for }1\leq r\leq b\}$ \\
nonisotropic & \\
\hline 
$\alpha$ regular & $S(\a)^{(r_{\alpha})_{\smallbullet}}\prod\limits_{s=0}^{n_{\alpha}-1}(h_{\alpha}-k_{\alpha}(b+n_{\alpha}-1-2s))$ \\
\hline
\end{tabular}
\end{table}   
Notice that $S(\a)^{(r_{\alpha})_{\smallbullet}}=\C[h_{\alpha}(h_{\alpha}-2a)]$.  On the other hand, we have the explicit computations of $HC(\AA_{(\g,\k)})$ for all rank-one pairs from \cite[Thm. 9.1]{Sh2}, which we copy below for the reader's convenience.  We slightly modify the notation from \emph{loc. cit.} (e.g.~we change $\Bbbk$ to $\C$), and add relevant details.  

\begin{enumerate}
		\item[(i)] $(\o\s\p(1|2)\times\o\s\p(1|2),\o\s\p(1|2))$: $(\alpha,\alpha)=1$, $\alpha(t)=1$, and 
		\[
		HC(\AA_{(\g,\k)})=S(\a)^{(r_{\alpha})_{\smallbullet}}(t+\frac{1}{2}).
		\]
		\item[(ii)] $(\g\l(1|1)\times\g\l(1|1),\g\l(1|1))$: $\alpha$ isotropic, and
		\[
		HC(\AA_{(\g,\k)})=h_{\alpha} S(\a).
		\]
		\item[(iii)] $(\g\l(m|n),\g\l(m-1|n)\times\g\l(1))$: $(\alpha,\alpha)=1/2$, $t=h_{\alpha}$, and
		\[
		HC(\AA_{(\g,\k)})= \C[t(t-n+m-1)]\langle t(t-1)\cdots(t-(n-1))\rangle.
		\]
		\item[(iv)] $(\o\s\p(2|2n),\o\s\p(1|2n))$: $(\alpha,\alpha)=1$, $t=h_{\alpha}$, and
		\[
		HC(\AA_{(\g,\k)})=\{p\in \C[t]:p(n+r)=(-1)^{r}p(n-r):1\leq r\leq n\},
		\]
		\item[(v)] $(\o\s\p(m|2n),\o\s\p(m-1|2n))$, $m\geq3$: $(\alpha,\alpha)=1$, $t=h_{\alpha}$, and
		\[
		HC(\AA_{(\g,\k)})= \C[t(t-2n+m-2)]\langle(t-1)(t-3)\cdots(t-(2n-1))\rangle.
		\]
		\item[(vi)] $(\o\s\p(m|2n),\o\s\p(m|2n-2)\times\s\p(2))$, $n\geq 2$: $(\alpha,\alpha)=1/2$, $t=h_{\alpha}$, and
		\[
		HC(\AA_{(\g,\k)})= \C[t(t+2n-m-1)]\langle (t+1)t(t-1)\cdots(t-(m-2))\rangle.
		\]
	\end{enumerate}

To help the reader check that the listed computations of $HC(\AA_{(\g,\k)})$ actually agree with $M_{\a,\{\alpha\}}$ in each case, we provide below a table with the relevant data about the restricted root systems for each of the above cases.

\renewcommand{\arraystretch}{1}	

\begin{table}[h!]

\centering
\begin{tabular}{|c|c|c|c|c|c|c|}
\hline
Case & Type & $m_\alpha$ &$m_{2\alpha}$ & $\rho$ & $\rho^{\text{pure}}$\\
\hline
(i) & reg & $(0|2)$ & $(2|0)$ & $\alpha$ & $\alpha$ \\
\hline
(ii) & sing & $(0|2)$ & $(0|0)$ & $-\alpha$ & $-\alpha$\\
\hline
(iii) & reg & $(2(m-2)|2n)$ &$(1|0)$ & $(m-n-1)\alpha$ & $(1-n)\alpha$ \\
\hline
(iv) & sing & $(0|2n)$ &$(0|0)$ & $-n\alpha$ & $-n\alpha$ \\
\hline
(v) & reg & $(m-2|2n)$ &$(0|0)$ & $(m/2-1-n)\alpha$ & $-n\alpha$ \\
\hline
(vi) & reg & $(4(n-2)|2m)$&$(3|0)$ & $(2n-m-1)\alpha$ & $(3-m)\alpha$ \\
\hline
\end{tabular}
\end{table}

\end{proof}

\begin{cor}\label{cor ghost rank one simple inclusion}
    For any simple root $\alpha\in\Sigma$, we have
    \[
    HC_{\Sigma}(\AA_{(\g,\k)})\sub M_{\a,\Sigma,\{\alpha\}}.
    \]
\end{cor}
\begin{proof}
    We apply Proposition \ref{prop rank one ghost} along with the fact that $(\rho,\alpha)=(\rho_{\alpha},\alpha)$ for singular roots by Lemma \ref{lemma local global rho}, and $(\rho^{\text{pure}},\alpha)=(\rho^{\text{pure}}_{\alpha},\alpha)$ for $\alpha$ a regular root (using a reflection argument as in Lemma \ref{lemma local global rho}).
\end{proof}


\subsection{Changing positive systems}

\begin{prop}\label{prop ghost borel change}
Let $D\in\AA_{(\g,\k)}$ and $\alpha\in\Sigma$. One has
	\begin{enumerate}
		\item  If $\alpha$ is a regular root, then $HC_{r_{\alpha}\Sigma}(D)=HC_{\Sigma}(D)\circ r_{\alpha}$.
		\item  If $\alpha$ is a singular root, then $HC_{\Sigma}(D)=(-1)^{n_{\alpha}}t_{m_{\alpha,1}\alpha}^*(HC_{r_{\alpha}\Sigma}(D))$.  In particular,
        \[
        HC_{\Sigma}(\AA_{(\g,\k)})\sub t_{m_{\alpha,1}\alpha}^*(HC_{r_{\alpha}\Sigma}(\AA_{(\g,\k)})).
        \]
	\end{enumerate}
    
\end{prop}

\begin{proof}
The case (1) is proven in the exact same way as in the proof of Proposition \ref{prop: reflections on HC}, where we use that ghost distributions are $K_0^\circ$-invariant.

Now let us assume $\alpha$ is singular.  
As in the proof of Proposition \ref{prop: reflections on HC}, we may assume that $(\g(\alpha),\k(\alpha))$ is either $(\o\s\p(2|2n),\o\s\p(1|2n))$, or $(\g\l(1|1)\times\g\l(1|1),\g\l(1|1))$.    
Therefore, by density, it suffices to check our formula for $\lambda\in\Lambda^+$ which is typical.  
Take $\lambda\in\Lambda^+$ typical, and let $f_\lambda$ be the corresponding eigenfunction on $G/K$ satisfying $f_{\lambda}(eK)=1$ (see (\ref{eqn f lambda})).  Let $g$ be a lowest weight vector of $\UU\g(\alpha)f_{\lambda}$, and assume further that $g(eK)=1$.  By the Iwasawa decomposition, we may write $g=(k+1)f_{\lambda}$, where $k\in\UU^+\k(\alpha)$.   As already noted in Proposition \ref{prop: reflections on HC}, the weight of $g$ is $t_{-m_{\alpha,1}\alpha}(\lambda)=\lambda-m_{\alpha,1}\alpha$.  
	
	Write $A$ for the subgroup of $G_0$ integrating $\a\sub\g_{\ol{0}}$, and let $a\in A$ have eigenvalues $\pm\sqrt{-1}$ on $\g(\alpha)_{\ol{1}}$, and eigenvalues $\pm1$ on $\g(\alpha)_{\ol{0}}$.  We note this is possible since $\a$ contains a grading operator for $\g(\alpha)$. Then we see that $\Ad(a)$ interchanges $\k(\alpha)$ and $\k'(\alpha)$ (in fact it is an interlacing automorphism of $(\g(\alpha),\k(\alpha))$ in the language of \cite{Sh2}).  Thus we obtain that
	\[
	e^{\lambda-m_{\alpha,1}\alpha}(a)(g)=a(g)=(k'+1)e^{\lambda}(a)f_{\lambda},
	\]
    for some $k'\in\UU^+\k(\alpha)\k(\alpha)$.  Therefore,
	\[
	e^{-m_{\alpha,1}\alpha}(a)g=(k'+1)f_{\lambda},
	\]
	where $k'\in\UU^+\k'$. 
    If we apply $D\in\AA_{(\g,\k)}$ to both sides, and use that $D$ is $\k'$-invariant by definition, we obtain that
	\[
	e^{-m_{\alpha,1}\alpha}(a)HC_{r_{\alpha}\Sigma}(D)(\lambda-m_{\alpha,1}\alpha)=HC_{\Sigma}(D)(\lambda).
	\]
	By definition $e^{\alpha}(a)=\pm\sqrt{-1}$; since $m_{\alpha,1}=2n_{\alpha}$, we obtain our desired formula:
	\[
	(-1)^{n_{\alpha}}HC_{r_{\alpha}\Sigma}(D)(t_{-m_{\alpha,1}\alpha}(\lambda))=HC_{\Sigma}(D)(\lambda).
	\]

\end{proof}

\begin{cor}\label{cor invce for HC ghost}
    $HC_{\Sigma}\left(\AA_{(\g,\k)}\right)\sub M_{\a,\Sigma}$.
\end{cor}
\begin{proof}

    The proof is verbatim to Corollary \ref{cor invce for HC}, where Lemma \ref{lemma local desc M_a} is used along with Corollary \ref{prop ghost borel change}, Corollary \ref{cor ghost rank one simple inclusion} and Lemma \ref{lemma trans local invce condition}.
\end{proof}

\subsection{Bijectivity for the ghost case}
We show that the Harish-Chandra map is surjective by showing that  $M_\a$ and $HC(\AA_{(\g,\k)})$ have the same associated graded algebra.

Let
$$N_\a:=S(\a)^{W}\prod\limits_{\alpha\in\Delta_{odd}^+}h_{\alpha}^{n_{\alpha}}.$$
This space is an $I_\a$-module and the product of two elements in $N_\a$ belongs to $I_\a$.
Note that $N_\a$ does not depend on the choice of $\Sigma$. 
\begin{prop}
One has
$\operatorname{gr}M_{\a}=N_\a$. 
In particular, $HC: \AA_{(\g,\k)} \rightarrow M_\a$ is surjective.
\end{prop}


\color{black}

\begin{proof}
By (\ref{eq:graded dim}), we have $\gr HC(\AA_{(\g,\k)})=N_\a$ and by Corollary \ref{cor invce for HC ghost}, it follows that $HC(\AA_{(\g,\k)})\subseteq M_\a$. 
We are left to show that $\gr M_\a\subseteq N_\a$.

For $f\in S(\a)$, we denote by $\bar f$ the leading term of $f$. 
Let $f\in M_\a$. 
We prove that $\bar f$ is $W$-invariant and  divisible by $\prod\limits_{\alpha\in\Delta_{odd}^+}h_{\alpha}^{n_{\alpha}}$, which itself is $W$-invariant.
By definition of $M_\a$, we can write $f=T_\Sigma\cdot g$ where  $g\in S(\a)$ is $W_{\smallbullet}$-invariant. 
Hence $\bar f=\overline{T_\Sigma}\cdot \bar g$ where  both $\overline{T_\Sigma}$ and $ \bar g$ are  $W$-invariant. Indeed, $$\overline{T_\Sigma}=\prod\limits_{\alpha \in\Delta_{odd}\setminus\Delta^+_{sing, noniso}}h_{\alpha}^{n_{\alpha}}$$
which is $W$-invariant. 
We are left to show divisibility by $h_\alpha^{n_{\alpha}}$ when $\alpha$ is singular non-isotropic. 
By \cref{lemma conj props}(iii),  it is enough to show it for $\alpha$ admissible.



Note that $f_\rho(\lambda):=f(\lambda-\rho)$ has the same leading term as $f(\lambda)$ and 
\[
f(\lambda+s\alpha)=(-1)^{s}f(\lambda-s\alpha),\quad1\le s\le n_\alpha
\]
for all $\lambda\in\left\{ \lambda\in\mathfrak{a}^{*}\mid\left(\lambda,\alpha\right)=0\right\} $.
We claim that  $\bar{f}_\rho$ is divisible by $h_\alpha^{n_\alpha}$.

Let $m=\operatorname{deg}(f)$, and write $f_\rho=\sum_{i}k_{m-i}h_{\alpha}^{i}$ where $k_{i}\in\C\left[\ker\alpha\right]$
and $\deg k_{j}\le j$. 
By assumption, for each $s=1,\ldots,n_\alpha$, we have
\begin{align*}
\sum_{i=0}^m k_{m-i}s^{i} & =(-1)^{s}\sum_{i=0}^m k_{m-i}\left(-s\right)^{i}
\end{align*}
 This is equivalent to 
\[
\sum_{i=1,3,\ldots\le m}k_{m-i}s^{i}=0\text{ for even }s,
\]

\[
\sum_{i=0,2,4,\ldots\le m}k_{m-i}s^{i}=0\text{ for odd }s.
\]
The first identity shows that $f_{\text{odd}}:=\sum_{i\text{ odd}}k_{m-i}h_{\alpha}^{i}$ vanishes at $h_\alpha=0,\pm 2,\pm4,\ldots,\pm(2\left\lfloor \frac{n_\alpha}{2}\right\rfloor)$ and so takes the form $$f_{\text{odd}}=
K_{\text{odd}}
h_\alpha(h_\alpha^2-2)(h_\alpha^2-4)\cdots(h_\alpha^2-2\left\lfloor \frac{n_\alpha}{2}\right\rfloor)$$
for some $K_{\text{odd}}$ in $S(\ker\alpha)$.
The second identity shows that $f_{\text{even}}:=\sum_{i\text{ even}}k_{m-i}h_{\alpha}^{i}$ vanishes at $h_\alpha=\pm 1,\pm3,\ldots,\pm(2\left\lceil \frac{n_\alpha}{2}\right\rceil-1)$, and so takes the form 
$$f_{\text{even}}
=K_{\text{even}}
(h_\alpha^2-1)(h_\alpha^2-3)\cdots (h_\alpha^2-(2\left\lceil \frac{n_\alpha}{2}\right\rceil-1))$$
for some $K_{\text{even}}$ in $S(\ker\alpha)$. Thus, the leading term of $f_\rho=f_{\text{odd}}+f_{\text{even}}$ must be divisible by $h_\alpha^{n_\alpha}$, completing the proof.
\end{proof}

\section{Appendix}

The following proposition is independent of the rest of the paper, but is given for the general understanding of invariant distributions. 
\begin{prop}
    Suppose $\s\mathfrak t\r(\operatorname{ad}x)=0$ for any $x\in\k$. Then $(\UU\g)^\k\cap\k\UU\g=(\UU\g)^\k\cap\UU\g\,\k$.
\end{prop}
\begin{proof}
    Consider again the symmetrization map $\lambda:S(\p)\to\UU\g$ which is $\k$-equivariant under the adjoint action. 
    Then we have the following decomposition of $\operatorname{ad}\k$-modules:
    $$\k\UU\g=\k\UU\g\,\k\oplus\k S(\p),\quad \UU\g\k=\k\UU\g\,\k\oplus S(\p)\k.$$
    It suffices to show that 
    $\left(\k S(\p)\right)^\k=\left(S(\p)\k\right)^\k.$
    Let $u\in\left(\k S(\p)\right)^\k$.
    One has 
    $$\left(\k S(\p)\right)^\k
    \cong \left(\k\otimes S(\p)\right)^\k
    \cong \Hom_\k(\mathbb C,\k\otimes S(\p))
    \cong \Hom_\k(\k^*,S(\p)).$$ 

This implies that any $u$ can be written in the form $\sum_i e_i\varphi_i$ where $e_1,\dots,e_n$ is a basis of the ideal $\mathfrak l:=(\ker u)^\perp \subset \k$ (if $u$ is viewed as a homomorphism, all the elements in its kernel vanish on $\mathfrak l$), and $\varphi_1,\ldots,\varphi_n$ is a dual basis of $\mathfrak l^*\subset S(\p)$. Here the duality is taken with respect to $\langle \varphi_i,e_j\rangle:=(u^{-1}(\varphi_i))(e_j)$.
    Note that $e_i$ and $\varphi_i$ have the same parity.
    
    We claim that $\sum_{i=1}^n(\operatorname{ad}^* e_i)(\varphi_i)=0$. 
    Indeed, for any $j=1,\ldots,n$, if $p(e_j)=0$, we have
    \begin{align*}
        \left\langle \sum_{i=1}^n(\operatorname{ad}^* e_i)(\varphi_i),e_j\right\rangle
        &=-\sum_{i=1}^n(-1)^{p(e_i)}\left\langle \varphi_i,(\operatorname{ad} e_i)e_j\right\rangle\\&=
        \sum_{i=1}^n(-1)^{p(e_i)}\left\langle \varphi_i,(\operatorname{ad} e_j)e_i\right\rangle\\&
        =\s\mathfrak t\r\operatorname{ad} e_j=0.
    \end{align*}
    For $e_j$ odd, $ \left\langle (\operatorname{ad}^* e_i)(\varphi_i),e_j\right\rangle=0$ for each $i=1,\ldots,n$ since $(\operatorname{ad}^* e_i)(\varphi_i)$ is an even element.
    Therefore, 
       $$u= \sum_i e_i\varphi_i=\sum_i \operatorname{ad}  e_i(\varphi_i)+(-1)^{p(e_i)}\varphi_i e_i=\sum_i (-1)^{p(e_i)}\varphi_i e_i$$
    and we get that $u\in\left(S(\p)\k \right)^\k$, as required.
\end{proof}

	\bibliographystyle{amsalpha}

    \textsc{\footnotesize S.R.: Bar-Ilan University, Ramat Gan, Israel
 }
 
\textit{\footnotesize Email address:}  \texttt{\footnotesize shifra.reif@biu.ac.il}

    \textsc{\footnotesize S.S.: Department of Mathematics, Rutgers University,paper 110 Frelinghuysen Rd, Piscataway, NJ 08854-8019, USA }
    
\textit{\footnotesize Email address:}  \texttt{\footnotesize sahi@math.rutgers.edu}

    \textsc{\footnotesize V.S.:  Deptartment of Mathematics, University of  California at Berkeley, Berkeley, CA 94720 }
    
    \textit{\footnotesize Email address:} \texttt{\footnotesize serganov@math.berkeley.edu}

	\textsc{\footnotesize A.S.: School of Mathematics and Statistics, University of Sydney, NSW 2006, Australia;} 	\textsc{\footnotesize School of Mathematics and Statistics, UNSW Sydney, NSW 2052, Australia}
	
	\textit{\footnotesize Email address:} \texttt{\footnotesize xandersherm@gmail.com}
    
\end{document}